\documentclass[11pt]{article} 

\usepackage[latin1,utf8]{inputenc}
\usepackage[T1]{fontenc} 
\usepackage{textcomp} 
\usepackage[english]{babel}
\usepackage[autolanguage]{numprint} 
\usepackage{hyperref} 
\usepackage{graphicx} 
\usepackage{subcaption}
\usepackage{epstopdf}
\usepackage{verbatim} 
\usepackage[all]{xy}
\usepackage{float}

\usepackage{lmodern}
\usepackage{mathrsfs}

\usepackage{amsmath,amssymb,amsthm} 
\usepackage{tikz,tkz-tab} 

\usepackage{mathdots}
\usepackage{makeidx}
\usepackage{stmaryrd} 
\usepackage{appendix}

\usepackage{enumitem}
\setenumerate[0]{label=(\alph*)}

\renewcommand\epsilon\varespilon 
\renewcommand\phi\varphi 
\renewcommand \subset \subseteq

\newcommand\AND{\quad\textrm{and}\quad}

\newcommand\be{\mathbf{e}}

\newcommand\bx{\mathbf{x}}  
\newcommand\bX{\mathbf{X}}
\newcommand\bY{\mathbf{Y}}
\newcommand\bZZ{\mathbf{Z}}

\newcommand\by{\mathbf{y}}  
\newcommand\bz{\mathbf{z}}  
\newcommand\bQ{\mathbb{Q}} 
\newcommand\bN{\mathbb{N}} 
\newcommand\bR{\mathbb{R}} 
\newcommand\bZ{\mathbb{Z}} 
\newcommand\codim{\mathrm{codim}}
\newcommand\cA{\mathcal{A}}
\newcommand\cB{\mathcal{B}}

\newcommand\cD{\mathfrak{D}}
\newcommand\cI{\mathcal{I}}

\newcommand\cU{\mathcal{U}}
\newcommand\cV{\mathcal{V}}

\newcommand\ee{\varepsilon}

\newcommand\GL{\mathrm{GL}}
\newcommand\GrO{\mathcal{O}} 
\newcommand\hlambda{\widehat{\lambda}} 
\newcommand\homega{\widehat{\omega}} 

\newcommand\tP{\widetilde{P}}
\newcommand\tQ{\widetilde{Q}}
\newcommand\tR{\widetilde{R}}
\newcommand\tS{\widetilde{S}}
\newcommand\HH{\mathcal{H}} 
\newcommand\HHH{\mathcal{D}_\xi'}
\newcommand\HHstar{\mathcal{D}_{\xi}}

\newcommand\ie{\textsl{i.e. }} 
\newcommand\eg{\textsl{e.g. }} 

\newcommand{\norm}[1]{\left\|#1 \right\|_2} 
\newcommand{\normH}[1]{\|#1\|} 
\newcommand{\psc}[2]{\left(#1 \mid #2\right)} 

\newcommand\la{\langle}
\newcommand\ra{\rangle}
\newcommand{\Vect}[2][]{\left\la #2\right\ra_{#1}} 

\newcommand\Res{\mathrm{Res}}

\theoremstyle{definition} 
\newtheorem{Def}{Definition}[section]

\theoremstyle{plain} 
\newtheorem{Prop}[Def]{Proposition} 
\newtheorem{Lem}[Def]{Lemma} 
\newtheorem{Thm}{Theorem}[section] 
\newtheorem{Cor}[Def]{Corollary} 

\theoremstyle{remark} 
\newtheorem{Rem}[Def]{Remark} 

\numberwithin{equation}{section} 

\title{On uniform polynomial approximation}  
\author{Anthony Poëls} 
\date{}

\newcommand{\Addresses}{{
  \bigskip
  \footnotesize
  \noindent\textsc{Universite Claude Bernard Lyon 1, \\
  Institut Camille Jordan UMR 5208, \\
  69622 Villeurbanne,\\
  France}
  \medskip
  \par\nopagebreak
  \noindent\textit{E-mail address}: \texttt{poels@math.univ-lyon1.fr}
}}

\newcommand{\MSC}{{
  \footnotesize
  \textbf{MSC~2020}: 11J13(Primary), 11J82 (Secondary).
}}

\newcommand{\keysW}{{
  \footnotesize
  \textbf{Keywords}: Exponent of Diophantine approximation, heights, uniform polynomial approximation.
}}

\newcommand{\Ack}{{
  \footnotesize

  \textbf{Acknowledgements}: I would like to thank Damien Roy for his attentive reading of this work and his many comments, which helped to improve its overall presentation and clarity.
}}

\begin{document} 

\baselineskip=17pt 

\maketitle

\begin{abstract}
    Let $n$ be a positive integer and $\xi$ a transcendental real number. We are interested in bounding from above the uniform exponent of polynomial approximation $\homega_n(\xi)$. Davenport and Schmidt's original 1969 inequality $\homega_n(\xi)\leq 2n-1$ was improved recently, and the best upper bound known to date is $2n-2$ for each $n\geq 10$. In this paper, we develop new techniques leading us to the improved upper bound $2n-\frac{1}{3}n^{1/3}+\GrO(1)$.
\end{abstract}

\MSC

\keysW

\section{Introduction}

Let $\xi$ be a non-zero real number and let $n$ be a positive integer. Dirichlet's theorem (1842) is one of the most basic results of Diophantine approximation. It shows that for any real number $H > 1$, there exists a non-zero integer point $(x_0,\dots,x_n)\in\bZ^{n+1}$ such that
\begin{align}
    \label{eq: Dirichlet system}
    \max\big\{|x_1|,\dots,|x_n|\big\} \leq H \AND |x_0 + x_1\xi + \cdots + x_n\xi^n| \leq H^{-n}.
\end{align}
It is natural to ask if we can improve the exponent $n$ of $H^{-n}$, and this question gives rise to two Diophantine exponents. The so-called \textsl{uniform exponent of approximation} $\homega_n(\xi)$ (resp. the \textsl{ordinary} exponent $\omega_n(\xi)$), is the supremum of the real numbers $\omega > 0$ such that the system
\begin{align*}
    \normH{P} \leq H \AND 0 <|P(\xi)| \leq H^{-\omega}
\end{align*}
admits a non-zero solution $P\in\bZ[X]$ of degree at most $n$ for each sufficiently large $H$ (resp. for arbitrarily large $H$). Here, $\normH{P}$ denotes the (naive) \textsl{height} of $P$, defined as the largest absolute value of its coefficients. These quantities have been extensively studied over the past half-century, see for example \cite{bugeaud2015exponents} for a nice overview of the subject. By Dirichlet's theorem, if $\xi$ is not an algebraic number of degree $\leq n$, then we have
\[
    \omega_n(\xi) \geq \homega_n(\xi) \geq n,
\]
and it is well known that those inequalities are equalities for almost all real numbers $\xi$ (w.r.t. Lebesgue measure). Note that if  $\xi$ is an algebraic number of degree $d$, then $\homega_n(\xi)$ and $\omega_n(\xi)$ are both equal to $\min\{n,d-1\}$ (it is a consequence of Schmidt's subspace theorem, see \cite[Theorem 2.10]{bugeaud2015exponents}). We can therefore restrict our study to the set of transcendental real numbers. The initial question ``can we improve the exponent $n$ in Dirichlet's Theorem?'' may be rephrased as follows: ``does there exist a transcendental real number $\xi$ satisfying $\homega_n(\xi) > n$?''. For $n=1$ the answer is negative and rather elementary to prove, so the first non-trivial case is $n=2$. Before the early $2000$s, it was conjectured that no such number existed. This belief was swept away by Roy's extremal numbers \cite{roy2003approximation}, \cite{roy2004approximation}, \cite{arbourRoyCriterionDegreTwo}, whose exponent $\homega_2$ is equal to the maximal possible value $(3+\sqrt 5)/2 = 2.618\cdots$. Since then, several families of transcendental real numbers whose uniform exponent $\homega_2$ is greater than $2$ have been discovered (see for example \cite{roy2007two}, \cite{bugeaud2005exponentsSturmian}, \cite{poels2017exponents, poelsExpoGeneralClass2021}). However, for $n \geq 3$ the mystery remains, and it is still an open question wether or not there exists $\xi\in\bR\setminus\overline{\bQ}$ with $\homega_n(\xi) > n$.

\medskip

In this paper, we are interested in finding an upper bound for the uniform exponent $\homega_n(\xi)$, as this could provide clues to solving the initial problem. Brownawell's version of Gel'fond's criterion \cite{brownawell1974sequences} implies that $\homega_n(\xi) \leq 3n$. In 1969, Davenport and Schmidt \cite[Theorem 2b]{davenport1969approximation} showed that for any transcendental real number $\xi$ and any integer $n\geq 2$, we have
\begin{equation}
    \label{eq: estimation Davenport and Schmidt}
    \homega_n(\xi)\leq 2n-1.
\end{equation}
Up to now, few improvements have been made. Bugeaud and Schleischitz \cite[Theorem 2.1]{bugeaudSchleischitz2016Uniform} first got the upper bound
\begin{align}
    \label{eq: intro Bugeaud-Schleischitz}
    \homega_n(\xi) \leq n-\frac{1}{2} + \sqrt{n^2-2n+1/4} = 2n-\frac{3}{2}+ \ee_n,
\end{align}
where $\ee_n>0$ tends to $0$ as $n$ tends to infinity. Recently, Marnat and Moshchevitin \cite{marnat2018optimal} proved an important conjecture of Schmidt and Summerer on the ratio $\homega_n(\xi)/\homega_n(\xi)$ (also see \cite[Chapter 2]{PhDMartin2019} for an alternative proof based on parametric geometry of numbers). In \cite{schleischitz2017some}, Schleischitz pointed out that we can use the aforementioned inequality in the proof of \eqref{eq: intro Bugeaud-Schleischitz} to get
\[
    \homega_n(\xi) \leq 2n-2,
\]
for each $n\geq 10$. This is currently the best known upper bound. Let us also mention that using parametric geometry of numbers, Schleischitz \cite[Theorem 1.1]{schleischitz2017uniformPoly} was able to replace the estimate \eqref{eq: intro Bugeaud-Schleischitz} by
\begin{align*}
    \homega_n(\xi) \leq \frac{3(n-1)+\sqrt{n^2-2n+5}}{2} = 2n-2+ \ee_n'
\end{align*}
where $\ee_n' > 0$ tends to $0$ as $n$ tends to infinity. For $n=3,\dots,9$, bounds that are better than \eqref{eq: estimation Davenport and Schmidt}, but (strictly) greater than $2n-2$, are known. For example, for each transcendental real number $\xi$, we have
\[
    \homega_3(\xi)\leq 3+\sqrt 2 = 4.41\cdots
\]
see \cite{bugeaudSchleischitz2016Uniform}. In this paper, without relying on Marnat-Moshchevitin's inequality, we show in Section~\ref{Section: cas d = 2} that the upper bounds $\homega_n(\xi)\leq 2n-2$ holds for any $n\geq 4$. We also improve the upper bound for~$\homega_3$.

\begin{Thm}
    \label{Thm : main d=2}
   Let $n\geq 3$ be an integer and $\xi\in\bR$ be a transcendental real number. If $n\geq 4$, then
    \begin{align*}
        \homega_n(\xi) \leq 2n-2.
    \end{align*}
    For $n=3$, we  have the weaker estimate $\homega_3(\xi)\leq 2+\sqrt 5 = 4.23\cdots$.
\end{Thm}

We do not think that these upper bounds are optimal. Our main result below is a significant improvement of the previous results as $n$ tends to infinity and does not require Marnat and Moshchevitin's inequality \cite{marnat2018optimal}.
\begin{Thm}
    \label{Thm : main}
    Set $a=1/3$. There exists a computable constant $N\geq 1$ such that, for each $n\geq N$ and any transcendental real number $\xi\in\bR$, we have
    \begin{align*}
        \homega_n(\xi) \leq 2n- an^{1/3}.
    \end{align*}
\end{Thm}
The constant $a=1/3$ is not optimal. Numerical calculations based on the results from Section~\ref{section: proof of main thm} suggest that we could take $N$ rather ``small'' in Theorem~\ref{Thm : main} (maybe $N\leq 10^4$?). However, as we wanted to keep the arguments and calculations as clear and simple as possible, we did not try to provide an explicit value of $N$. 

\medskip

Theorem~\ref{Thm : main} can be compared to \cite[Theorem~1.1]{poels2022simultaneous}, where we study $\hlambda_n(\xi)$, the uniform exponent of rational simultaneous approximation to the successive powers $\Xi=(1,\xi,\xi^2,\dots,\xi^n)$ (which is known to be, in a sense, dual to $\homega_n(\xi)$), see Section~\ref{section: notation} for the precise definition and more details. We were not able to deduce one result from the other, even though some objects from the proofs clearly play similar roles. For example, given a polynomial $P\in\bZ[X]$ of degree at most $n$, which is a good approximation, we can associate the $k+1$ polynomials $P,XP,\dots,X^{k}P$ of degree at most $n+k$. They provide information on $\homega_{n+k}(\xi)$. On the other hand, if we consider $\by\in\bZ^{n+1}$ which is a good approximation of $\Xi$ (for simultaneous approximation), we can associate the $k+1$ blocks of successive $n+1-k$ coordinates of $\by$, which are rather good approximations of $(1,\xi,\dots,\xi^{n-k})$. They in turn provide information on $\hlambda_{n-k}(\xi)$. Note that the difficulties in the proofs of both theorems are not in the same places. In particular, in this paper we have to work with \textsl{irreducible} polynomials, a rather heavy constraint. Also, one of the most delicate parts of our approach is to bound from above the ordinary exponent $\omega_n(\xi)$, whereas this is rather ``simple'' to do that for the ordinary exponent $\lambda_n(\xi)$ in \cite{poels2022simultaneous}.

\medskip

Before presenting our strategy, let us quickly explain Davenport and Schmidt's proof of the upper bound \eqref{eq: estimation Davenport and Schmidt}. Given a real number $\homega < \homega_n(\xi)$, they show, using elementary means and Gelfond's Lemma, that there are infinitely many pairs of coprime polynomials $P,Q\in \bZ[X]$ of degree at most $n$, such that
\[
    \normH{Q} \leq \normH{P} \AND \max\{|Q(\xi)|, |P(\xi)|\} \ll \normH{P}^{-\homega},
\]
(where the implicit constant only depends on $n$). It implies that the resultant $\Res(P,Q)$, which is a non-zero integer, satisfies
\[
    1 \leq |\Res(P,Q)| \ll \normH{P}^{n-1}\normH{Q}^{n-1} \max\big\{\normH{P}|Q(\xi)|, \normH{Q}|P(\xi)| \big\} \ll \normH{P}^{2n-1-\homega}.
\]
The first upper bound for $|\Res(P,Q)|$ is classical, see Lemma~\ref{lem: estimation classique du resultant}. Since $\normH{P}$ can be arbitrarily large, they deduced that the exponent $2n-1-\homega$ is non-negative. Estimate \eqref{eq: estimation Davenport and Schmidt} follows by letting $\homega$ tend to $\homega_n(\xi)$. Note that the term $2n$ in \eqref{eq: estimation Davenport and Schmidt} is directly related to the size of the $2n\times2n$ determinant defining $\Res(P,Q)$ (if we suppose that $P$ and $Q$ have degree exactly $n$).

\medskip

The key idea in the proof of our main Theorem~\ref{Thm : main}  is to work with a large number of ``good'' linearly independent polynomial approximations $Q_0, \dots Q_{j+1}$ rather than just two polynomials $P$ and $Q$ as above. By doing this, we can replace $\Res(P,Q)$ by a non-zero $(2n-j)\times (2n-j)$ determinant (whose non-zero entries are among the coefficients of the polynomials $Q_k$). Under the ideal and unlikely assumption that
\begin{align}
    \label{eq intro: controle des Q_i}
    \normH{Q_k} \leq \normH{Q_0} \AND |Q_k(\xi)| \ll \normH{Q_0}^{-\homega} \qquad \textrm{(for $k=0,\dots,j$)},
\end{align}
the aforementioned determinant would be bounded from above by $\normH{Q_0}^{2n-j-1 - \homega}$. So, together with an additional non-vanishing assumption, it would lead to $\homega_n(\xi) \leq 2n-j-1$. Several new difficulties arise when trying to make the above arguments work. We introduce the tools for the construction of the generalized resultant in Section~\ref{section: espaces V_N}. To ensure that this determinant does not vanish, we need the extra assumption that $Q_0,\dots,Q_{j+1}$ are irreducible polynomials. The idea is to first fix a sequence of best approximations, that we called \textsl{minimal polynomials}, and then to consider their highest-degree irreducible factors  (which also happen to be rather good approximations). We deal with this question in Section~\ref{section: sequence des Q_i}. Two obstacles remain. Firstly, note that it may be possible that the best polynomial approximations span a subspace of dimension~$3$, even when $\xi$ is transcendental and $n$ is large, see \cite[Theorem~1.3]{moshchevitin2007best}. Therefore, as soon as $j > 1$ (we will later choose $j\asymp n^{1/3}$), we have to justify that we can find $j+2$ linearly independent polynomials as above. The second major problem is the control of the sequence $Q_0,\dots,Q_{j+1}$. Estimates \eqref{eq intro: controle des Q_i} seem out of reach, instead we get upper bounds of the form
\begin{align}
    \label{eq intro: controle des Q_i V2}
     \normH{Q_k} \leq \normH{Q_0} \AND |Q_k(\xi)| \ll \normH{Q_0}|^{-\homega\theta} \qquad \textrm{(for $k=0,\dots,j$)},
\end{align}
where $\theta < 1$ depends only on $n$ and $j$, and is ``close'' to $1$ if $j$ is not too large compared to $n$. The main ingredients for showing this are related to \textsl{twisted heights}, see Sections~\ref{def petitesse HHstar(V)} and Appendix~\ref{subsection: twisted heights}, and an important inequality on the height of subspaces due to Schmidt. 
The parameter $\theta$ in \eqref{eq intro: controle des Q_i V2} is a function of the exponent of best approximation $\omega_n(\xi)$. We show in Section~\ref{section: estimation de omega_n} that if the uniform exponent satisfies $\homega_n(\xi) \geq 2n-d$ (with $d \ll  n^{1/3}$), then the ordinary exponent $\omega_n(\xi)$ is bounded from above by $2n+2d^2$, and the ratio $\homega_n(\xi)/\omega_n(\xi)$ is therefore close to~$1$. This part, which is essentially independent from the others, is rather delicate, because we work with the polynomials $Q_i$. They are certainly irreducible, but not as good approximations as the minimal polynomials. More precisely, there could be large gaps between the height of two successive $Q_i$. If we could drop the irreducibility condition and directly work with the sequence of minimal polynomials, we could possibly replace the upper bound $2n- \GrO(n^{1/3})$ with  $2n- \GrO(n^{1/2})$ in Theorem~\ref{Thm : main}. Section~\ref{section: proof of main thm} is devoted to the proof of Theorem~\ref{Thm : main}.

\section{Notation}
\label{section: notation}
Throughout this paper, $\xi$ denotes a transcendental real number.

\medskip

The floor (resp. ceiling) function is denoted by $\lfloor \cdot \rfloor$ (resp. $\lceil \cdot \rceil$). If $f, g : I \rightarrow [0, +\infty)$ are two functions on a set $I$, we write $f = \GrO(g)$ or $f \ll g$ or $g \gg f$ to mean that there is a positive constant $c$ such that $f (x) \leq cg(x)$ for each $x\in I$. We write $f \asymp g$ when both $f \ll g$ and $g \ll f$ hold. 

\medskip

Let $K$ be a field. If $\cA$ is a subset of a $K$-vector space $V$, we denote by $\Vect[K]{\cA} \subset V$ the $K$-vector space spanned by $\cA$, with the convention that $\Vect[K]{\emptyset} = \{0\}$.

\medskip

Given a ring $A$ (typically $A=\bR$ or $\bZ$) and an integer $n\geq 0$, we denote by $A[X]$ the ring of polynomials in $X$ with coefficients in $A$, and by $A[X]_{\leq n} \subset A[X]$ the subgroup of polynomials of degree at most $n$. We say that $P\in\bZ[X]$ is \textsl{primitive} if it non-zero and the greatest common divisor of its coefficients is $1$. Given $P=\sum_{k=0}^{n} a_kX^k \in \bR[X]$, we set
\begin{align*}
    \normH{P} = \max_{0\leq k \leq n} |a_k|.
\end{align*}

Gelfond's Lemma can be written as follows (see \eg \cite[Lemma A.3]{bugeaud2004approximation} as well as \cite{brownawell1974sequences}). For any non-zero polynomials $P_1,\dots,P_r\in\bR[X]$ with product $P = P_1\cdots P_r$ of degree at most $n$, we have
\begin{equation}
    \label{eq: Gelfond's Lemma}
    e^{-n} \normH{P_1}\cdots \normH{P_r} < \normH{P} < e^{n} \normH{P_1}\cdots \normH{P_r}.
\end{equation}
In particular, for each non-zero polynomial $P \in \bZ[X]_{\leq n}$ and each factor $Q\in\bZ[X]$ of $P$, we have $e^{-n}\normH{Q} < \normH{P}$. We will often use \eqref{eq: Gelfond's Lemma} as follows. If $Q\in\bZ[X]_{\leq n}$ is irreducible and if $P\in\bZ[X]_{\leq n}$ is a non-zero polynomial which satisfies $\normH{P} \leq e^{-n}\normH{Q}$, then $Q$ cannot divide $P$. They are thus coprime polynomials.

\medskip

We recall the definition of the resultant, which, as explained in the introduction, is useful for estimating the exponent $\homega_n(\xi)$ (also see Section~\ref{section: resultant et premier resultat}). Let $P,Q\in\bZ[X]$ be non-constant polynomials of degree $p$ and $q$ respectively, and let $a_i,b_j\in \bZ$ such that $P(X) = \sum_{k=0}^{p} a_kX^k$ and $Q(X) = \sum_{k=0}^{q} b_kX^k$.
Their \textsl{resultant} $\Res(P,Q)$ is defined as the $(q+p)$-dimensional determinant
\begin{align}
    \label{eq det resultant}
    \Res(P,Q) = \begin{array}{cc}
        \left | \begin{array}{cccc}
                  a_p & 0 & \dots \\
                  a_{p-1} & a_p  &     \\
                  \vdots & \vdots & \ddots \\
                  a_0    &   \\
                  0      &  a_0 \\
                  \vdots &  \vdots & \ddots  \\
                         &         &        & a_0
        \end{array}\right. &
        \left. \begin{array}{cccc}
                  b_q & 0 & \dots \\
                  b_{q-1} & b_q  &     \\
                  \vdots & \vdots & \ddots \\
                  b_0    &   \\
                  0      &  b_0 \\
                  \vdots &  \vdots & \ddots  \\
                         &         &        & b_p
        \end{array}\right|  \\
    \underbrace{\hspace*{2.7cm}}_{q}   & \underbrace{\hspace*{2.7cm}}_{p}
    \end{array}.
\end{align}

Besides the exponents of linear approximation $\omega_n$ and $\homega_n$, we will also need the following exponents of simultaneous rational approximation. For each positive integer $n$, the exponent $\hlambda_n(\xi)$ (resp. $\lambda_n(\xi)$) is the supremum of the real numbers $\lambda \geq 0$ such that the system
\begin{align*}
    |y_0| \leq Y \AND L(\by) \leq Y^{-\lambda} \quad \textrm{where }  L(\by):=\max_{1\le k \le n} |y_0\xi^k - y_k|,
\end{align*}
admits a non-zero integer solution $\by=(y_0,\dots,y_n)\in\bZ^{n+1}$ for each sufficiently large $Y\geq 1$ (resp. for arbitrarily large $Y$). Dirichlet's theorem \cite[\S II.1, Theorem 1A]{schmidt1996diophantine} implies that $\hlambda_n(\xi)\geq 1/n$. The best upper bounds known to date for $\hlambda_n(\xi)$ when $n\geq 4$ are established in a join work with Roy in \cite{poels2022simultaneous}. In particular, there is an explicit positive constant $a$ such that
\[
    \hlambda_n(\xi) \leq \frac{1}{n/2+an^{1/2}+1/3},
\]
and sharper results are also obtained when $n$ is small.

\section{Minimal polynomials}
\label{section: minimal pol}

A \textsl{sequence of minimal polynomials} (associated to $n$ and $\xi$) is a sequence $(P_i)_{i\geq 0}$ of non-zero polynomials in $\bZ[X]_{\leq n}$ satisfying the following properties
\begin{enumerate}[label=(\roman*)]
  \item the sequence $\big(\normH{P_i}\big)_{i\geq 0}$ is strictly increasing,
  \item the sequence $\big(|P_i(\xi)|\big)_{i\geq 0}$ is strictly decreasing,
  \item if $|P(\xi)| < |P_i(\xi)|$ for some index $i\geq 0$ and a non-zero $P\in\bZ[X]_{\leq n}$, then $\normH{P}\geq \normH{P_{i+1}}$.
\end{enumerate}

Note that if we require the dominant coefficient of $P_i$ to be positive (and since $\xi$ is transcendental), then the above sequence is uniquely determined (up to the first terms). Let $(P_i)_{i\geq 0}$ be a sequence as above. We have the classical formulas:
\begin{align}
    \label{eq: exposant via minimal points}
    \homega_n(\xi) = \liminf_{i\rightarrow \infty} \frac{-\log |P_i(\xi)|}{\log \normH{P_{i+1}}} \AND \omega_n(\xi) = \limsup_{i\rightarrow \infty} \frac{-\log |P_i(\xi)|}{\log \normH{P_{i}}}.
\end{align}
In particular, given a positive real number $\homega$ with $\homega < \homega_n(\xi)$, then we have, for each sufficiently large index $i$,
\begin{equation}
    \label{eq: H(P_i+1) controle par H(P_i)}
    |P_i(\xi)| \leq \normH{P_{i+1}}^{-\homega} \AND \normH{P_{i+1}}^\tau \leq \normH{P_i},\quad \textrm{where } \tau:= \frac{\homega}{\omega_n(\xi)},
\end{equation}
(with the convention $\tau = 0$ if $\omega_n(\xi) = \infty$). The second inequality in \eqref{eq: H(P_i+1) controle par H(P_i)} justifies
the need to bound from above $\omega_n(\xi)$. Given a non-zero $P \in \bZ[X]$, we set $\omega(P)=0$ if $\normH{P} = 1$. Otherwise, we denote by $\omega(P)$ the real number satisfying
\[
    |P(\xi)| = \normH{P}^{-\omega(P)}.
\]
With this notation, we have
\begin{align}
    \label{eq: encadrement omega(P_i)}
    \omega_n(\xi) = \limsup_{\substack{\normH{P}\rightarrow \infty \\ P\in\bZ[X]_{\leq n}}} \omega(P) = \limsup_{i\rightarrow \infty} \omega(P_i) \AND \liminf_{i\rightarrow \infty} \omega(P_i) \geq \homega_n(\xi).
\end{align}

The following results are well-known. We prove them for the sake of completion. The first one follows from the arguments of the proof of \cite[Lemma 2]{davenport1967approximation} (also see \cite[Lemma~4.1]{roy2004approximation}).

\begin{Lem}
    \label{Lem: P_i, P_i+1 est une Z-base}
    Let $i\geq 0$ and write $V_i =  \Vect[\bR]{P_i,P_{i+1}} \subset \bR[X]_{\leq n}$. Then $\{P_i, P_{i+1}\}$ forms a $\bZ$--basis of the lattice $V_i\cap \bZ[X]_{\leq n}$.
\end{Lem}

\begin{proof}
     By contradiction, suppose that $\{P_i, P_{i+1}\}$ is not a $\bZ$--basis of $V_i\cap \bZ[X]_{\leq n}$. Then there exists a non-zero $Q\in \bZ[X]_{\leq n}$ which may be written as $Q = rP_i + sP_{i+1}$, where $r,s\in\bQ$ satisfy $|r|,|s|\leq 1/2$. In particular, we have
    \begin{align*}
        \normH{Q} \leq |r|\normH{P_i} + |s|\normH{P_{i+1}} < \normH{P_{i+1}} \AND |Q(\xi)| \leq |r||P_i(\xi)| + |s||P_{i+1}(\xi)| < |P_{i}(\xi)|.
    \end{align*}
    This contradicts the minimality property of $P_i$.
\end{proof}

The next result is analogous to the second part of \cite[Lemma~4.1]{roy2004approximation}. The construction of $S_i$ is also somewhat similar to what Davenport and Schmidt do in \cite{davenport1967approximation}.

\begin{Lem}
    \label{Lem: P_iH_i+1 = P_j-1H_j}
    For each $i\geq 0$, define
    \[
        S_i = P_i(\xi)P_{i+1} - P_{i+1}(\xi)P_i \in\bR[X]_{\leq n}.
    \]
    Then
    \[
        \frac{1}{2}\normH{S_i} \leq \normH{P_{i+1}}|P_i(\xi)| \leq 2\normH{S_i}.
    \]
    Moreover, if for integers $0\leq i < j$ the space spanned by $P_i, P_{i+1}, \cdots, P_{j}$ has dimension~$2$, then $S_{j-1} = \pm S_i$. In particular
    \[
        \normH{P_{i+1}}|P_i(\xi)| \asymp \normH{P_{j}}|P_{j-1}(\xi)|.
    \]
\end{Lem}

\begin{Rem}
    Note that the quantity $\normH{S_i}$ satisfies $\normH{S_i} \asymp \HHstar(V_i)$, where $\HHstar$ is defined in Section~\ref{def petitesse HHstar(V)} and $V_i=\Vect[\bR]{P_i,P_{i+1}}$. We will study this function more deeply in full generality later.
\end{Rem}

\begin{proof}
    We easily get $\normH{S_i} \leq 2\normH{P_{i+1}}|P_i(\xi)|$. Define $R_+, R_- \in\bZ[X]_{\leq n}$ by
    \[
        R_\pm = P_{i+1} \pm P_i.
    \]
    Suppose that there exists $\ee\in\{+,-\}$ such that $|R_\ee(\xi)| \leq |P_i(\xi)|/2$. Then by minimality of $P_i$, we must have $\normH{R_\ee} \geq \normH{P_{i+1}}$. Since $S_i = P_i(\xi)R_\ee - R_\ee(\xi) P_i$, we find
    \begin{align*}
        \normH{S_i} \geq |P_i(\xi)|\normH{R_\ee} - |R_\ee(\xi)|\normH{P_i} \geq \frac{1}{2} \normH{P_{i+1}}|P_i(\xi)|.
    \end{align*}
    Assume that $|R_+(\xi)|, |R_-(\xi)| \geq |P_i(\xi)|/2$. This is equivalent to
    \[
        |P_{i+1}(\xi)| \leq \frac{1}{2}|P_i(\xi)|.
    \]
    Again, this yields $\normH{S_i} \geq |P_i(\xi)|\normH{P_{i+1}} - |P_{i+1}(\xi)|\normH{P_i} \geq \normH{P_{i+1}}|P_i(\xi)|/2$.

    \medskip

    Now, let us write $V_i = \Vect[\bR]{P_i,\dots,P_j}$, with $j > i$, and suppose that $V_i$ has dimension $2$. We need to prove that $S_{j-1} = \pm S_i$. If $j=i+1$ it is automatic, we may therefore assume that $j\geq i+2$. By Lemma~\ref{Lem: P_i, P_i+1 est une Z-base}, there exist $a,b\in\bZ$ such that $P_i = aP_{i+1} + bP_{i+2}$. Since $\{P_i,P_{i+1}\}$ is also a $\bZ$--basis of $V_i$, we have $b=\pm 1$, and we deduce that
    \[
        S_i = \big(aP_{i+1}(\xi) + bP_{i+2}(\xi)\big)P_{i+1} - P_{i+1}(\xi)\big(aP_{i+1} + bP_{i+2}\big) = -bS_{i+1} = \pm S_{i+1}.
    \]
    By induction, we get $S_i = \pm S_{i+1} = \cdots = \pm S_{j-1}$.
\end{proof}

The proof of \cite[Lemma~3]{davenport1967approximation} (which deals with the case $n=2$) yields the classical following result.

\begin{Lem}
    \label{lem: I infini}
    Suppose $n\geq 2$. Then, there are infinitely many indices $i\geq 1$ for which $P_{i-1}$, $P_i$ and $P_{i+1}$ are linearly independent.
\end{Lem}

\begin{proof}
    By contradiction, suppose that there exists $i\geq 0$ such that $V=\Vect[\bR]{P_i,P_{i+1},\dots}$ has dimension $2$. By Lemma~\ref{Lem: P_iH_i+1 = P_j-1H_j} there exists $c>0$ such that for each $j > i$ we have
    \begin{align*}
        0 < \normH{P_{i+1}}|P_i(\xi)| \leq c\normH{P_j}|P_{j-1}|.
    \end{align*}
    This leads to a contradiction since $\normH{P_j}|P_{j-1}| \leq \normH{P_j}^{1-\homega_n(\xi)+o(1)}$ tends to $0$ as $j$ tends to infinity.
\end{proof}

\begin{Rem}
    As mentioned in the introduction, it is however possible that all polynomials $P_i$ with $i$ large enough lie in a subspace of dimension $3$ , see \cite[Theorem~1.3]{moshchevitin2007best}.
\end{Rem}

\section{Resultant and first estimates}
\label{section: resultant et premier resultat}

The following useful result can be easily obtained from the proof of \cite[§5]{davenport1969approximation} (also see of \cite[Lemma 1]{brownawell1974sequences}). We recall the arguments since they illustrate (in a simpler situation) how we will deal with generalized determinants.

\begin{Lem}
    \label{lem: estimation classique du resultant}
    Let $p,q$ be positive integers with $p,q\leq n$. There exists a constant $c>0$ depending on $\xi$ and $n$ only, with the following property. For any polynomials $P,Q\in\bZ[X]$ of degree $p$ and $q$ respectively, we have
    \begin{equation*}
        |\Res(P,Q)| \leq c \normH{P}^{q-1}\normH{Q}^{p-1} \max\big\{\normH{P}|Q(\xi)|, \normH{Q}|P(\xi)| \big\}.
    \end{equation*}
\end{Lem}

\begin{proof}
    Let $a_i,b_j\in \bZ$ such that $P(X) = \sum_{k=0}^{p} a_kX^k$ and $Q(X) = \sum_{k=0}^{q} b_kX^k$. For $i=1,\dots,p+q-1$, we add to the last row of the determinant \eqref{eq det resultant} the $i$-th row multiplied by $\xi^{p+q-i}$. This last row now becomes
    \begin{align*}
        \Big(\xi^{q-1}P(\xi),   \dots,  \xi P(\xi), P(\xi), \xi^{p-1}Q(\xi), \dots, \xi Q(\xi),  Q(\xi)\Big).
    \end{align*}
    Using the upper bounds $|a_i| \leq \normH{P}$ and $|b_j|\leq \normH{Q}$ for the other entries of \eqref{eq det resultant}, we obtain
    \begin{align*}
        |\Res(P,Q)| \ll \normH{P}^{q-1}|P(\xi)| \normH{Q}^p + \normH{P}^q\normH{Q}^{p-1}|Q(\xi)|,
    \end{align*}
    where the implicit constant only depends on $p, q$ and $\xi$.
\end{proof}

The next result, which is also based on inequalities involving resultants, will be used in Section~\ref{section: estimation de omega_n}. It ensures that if $R\in\bZ[X]$ is a ``good'' approximation, in the sense that $R(\xi)$ is very small compared to $\normH{R}$, and if we write $R$ as a product of coprime polynomials $B_1\cdots B_k$, then one of those factors is also a ``good'' approximation, while the product of the others is not. 

\begin{Lem}
     \label{lem: lem general 2}
    Let $m,k$ be positive integers. There exists a constant $c>0$ depending on $m$ and $\xi$ only, with the following property. Let $B_1,\dots,B_k\in\bZ[X]$ be non constant, pairwise coprime polynomials, and suppose that $R:= B_1\cdots B_k$ has degree at most $m$. Then, there exists $j\in\{1,\dots,k\}$ such that
    \[
        |B_j(\xi)| \leq c\normH{R}^{m-1}|R(\xi)| \AND \prod_{\substack{i=1 \\ i\neq j}}^{k} |B_i(\xi)| \geq c^{-1}\normH{R}^{-(m-1)}.
    \]
\end{Lem}

\begin{proof}
     If $k=1$ this is trivial. We now suppose that $k\geq 2$ and we write $d_i = \deg(B_i)$ for $i=1,\dots, k$. By hypothesis, we have $\deg(R) = d_1+\cdots + d_k \leq m$. Note that
    \begin{align}
        \label{eq proof: log H(P) = sum log H(R_i)}
         |R(\xi)| = \prod_{i=1}^{k} |B_i(\xi)| \AND  \normH{R} \asymp \prod_{i=1}^{k} \normH{B_i},
    \end{align}
    the second inequality coming from Gelfond's (the implicit constants depending only on $m$). Let $j\in\{1,\dots,k\}$ be such that $|B_j(\xi)|$ is minimal and fix $i\in\{1,\dots,k\}$ with $i\neq j$. Since $B_i$ and $B_j$ are coprime, their resultant $\Res(B_i,B_j)$ is a non-zero integer. Using Lemma~\ref{lem: estimation classique du resultant}, we find
    \begin{align*}
        1 \leq |\Res(B_i,B_j)| & \ll \normH{B_i}^{d_j-1} \normH{B_j}^{d_i-1}\big( \normH{B_j}|B_i(\xi)| + \normH{B_i}|B_j(\xi)|\big) \\
        & \ll \normH{B_i}^{d_j} \normH{B_j}^{d_i}|B_i(\xi)|,
    \end{align*}
    the implicit constant depending only on $\xi$ and $m$, hence
    \[
        -\log |B_i(\xi)| \leq d_j\log \normH{B_i} + d_i\log \normH{B_j} + \GrO(1).
    \]
    On the other hand, by summing the above inequalities for $i\neq j$ (and since $\max\{m-d_j,d_j\} \leq m-1$), and by using \eqref{eq proof: log H(P) = sum log H(R_i)}, we obtain
    \begin{align*}
        \sum_{\substack{i=1 \\ i\neq j}}^{k} -\log |B_i(\xi)|
        & \leq d_j\sum_{\substack{i=1 \\ i\neq j}}^{k} \log \normH{B_i} +(m-d_j) \log \normH{B_j} + \GrO(1) \\
        & \leq (m-1)\log \normH{R} + \GrO(1).
    \end{align*}
    We easily deduce that
    \begin{align*}
        \prod_{\substack{i=1 \\ i\neq j}}^{k} |B_i(\xi)| \gg \normH{R}^{-(m-1)} \AND
         |R(\xi)| = \prod_{i=1}^{k}|B_i(\xi)| \gg |B_j(\xi)| \normH{R}^{- (m-1)}.
    \end{align*}
\end{proof}

\section{A sequence of irreducible polynomials}
\label{section: sequence des Q_i}

As explained in the introduction, to get the upper bound $\homega_n(\xi) \leq 2n-1$, the strategy of Davenport and Schmidt \cite{davenport1969approximation} consists in considering the resultant $\Res(P,Q)$ of two ``good'' polynomial approximations $P,Q\in \bZ[X]_{\leq n}$. To ensure that $\Res(P,Q)$ does not vanish, they need a polynomial $P$ which is irreducible (for it is then easy to find $Q$ so that $P$ and $Q$ are coprime). The same difficulty appears in \cite{bugeaudSchleischitz2016Uniform}. Similarly, we will not work directly with a sequence of minimal polynomials. Instead, we will considerer the largest irreducible factors of the minimal polynomials. Now, let $n,d$ be integers with
\[
    2\leq d < 1+ \frac{n}{2}.
\]
In this section, we assume that the transcendental real number $\xi$ satisfies $\homega_n(\xi) > 2n -d$ and we fix a real number $\homega$ (arbitrarily close to $\homega_n(\xi)$) such that
\begin{align}
    \label{eq: notation homega}
    \homega_n(\xi) > \homega > 2n -d.
\end{align}

We denote by $(P_i)_{i\geq 0}$ a sequence of minimal polynomials associated to $n$ and $\xi$. Our goal is to prove the existence of a sequence $(Q_i)_{i\geq 0}$ as below. 

\begin{Prop}
    \label{prop: existence Q_i}
    Suppose that \eqref{eq: notation homega} holds. Then, there exist a sequence $(Q_i)_{i\geq 0}$ of pairwise distinct polynomials in $\bZ[X]_{\leq n}$ and an index $j_0\geq 0$ with the following properties. The sequence $(\normH{Q_i})_{i\geq 0}$ is bounded below by $2$, unbounded and non-decreasing, and for any $i\geq 0$
    \begin{enumerate}[label=(\roman*)]
        \item \label{enum: property (Q_i) item 1} $Q_i$ is irreducible (over $\bZ$) and has degree at least $n-d+2$;
        \item $Q_i$ divides $P_j$ for some index $j\geq j_0$ (not necessarily unique), and for each $j\geq j_0$ there exists $k\geq 0$ such that $Q_k$ divides $P_j$;
        \item \label{enum: property (Q_i) item 6} $|Q_i(\xi)| = \normH{Q_i}^{-\omega(Q_i)}\leq \normH{Q_i}^{-\homega}$, and we further have 
            \begin{align}
                \label{eq: omega_n donné par les Q_i}
                \omega_n(\xi) =  \limsup_{k\rightarrow \infty} \omega(Q_k) \AND \liminf_{k\rightarrow \infty} \omega(Q_k) \geq \homega_n(\xi).
            \end{align}
        \item \label{enum: H(P_j) controle par H(Q_i)} if $Q_i$ divides a minimal polynomial $P_j$ with $j\geq j_0$, then
            \begin{equation}
                \label{eq: premiere estimation Pj vs Qi}
                \normH{P_j} \leq \normH{Q_i}^{1+\theta_i}, \quad \textrm{where } \theta_i = \frac{\omega(Q_i)-2n+d}{n-2d+3};
            \end{equation}
        \item \label{enum: H(Q_i+1) controle par H(Q_i)} we have
            \begin{equation}
                \label{eq: estimation Q_i+1 par Q_i}
                \normH{Q_{i+1}}^\tau \leq \normH{Q_i} \quad \textrm{where } \tau = \frac{\homega\big(\homega -n-d+3\big)}{\omega_n(\xi)\big(\omega_n(\xi)-n-d+3\big)},
            \end{equation}
            with the convention $\tau= 0$ if $\omega_n(\xi) = \infty$;
    \end{enumerate}
\end{Prop}

The above proposition is essentially a consequence of Lemma~\ref{lem: existence facteur irreductible grand} below. Assertion \ref{enum: property (Q_i) item 6} ensures that the polynomials $Q_i$ are quite good approximations, and they can be used to compute the exponent of best approximation $\omega_n(\xi)$. Estimate \eqref{eq: estimation Q_i+1 par Q_i} is the analog of the second inequality of \eqref{eq: H(P_i+1) controle par H(P_i)} but is way more difficult to prove. The main reason behind this difficulty is that there may be many polynomials $P\in\bZ[X]_{\leq n}$ with $\normH{Q_i} < \normH{P} < \normH{Q_{i+1}}$ and $|P(\xi)| < Q_i(\xi)$
\medskip

In order to prove Proposition~\ref{prop: existence Q_i}, we need the two technical lemmas below. Essentially, they will be used to prove that the factors of $P_i$ of small degree are bad approximations. This will lead to the existence of a factor of large degree which is necessarily a rather good approximation.

\begin{Lem}
   \label{lemme: omega(R) petit degre}
   Suppose that \eqref{eq: notation homega} holds. Then, there exists a constant $c\in(0,1)$ depending only on $\xi$ and $n$ such that for any non-zero polynomial $R\in \bZ[X]_{\leq n-d+1}$ we have
   \begin{equation}
        \label{lemme; eq1 : omega(R) petit degree}
        |R(\xi)|  \geq c\normH{R}^{-(n + \deg(R) - 1)} \geq c\normH{R}^{-(2n-d)}.
   \end{equation}
   In particular \eqref{lemme; eq1 : omega(R) petit degree} holds for any $R\in\bZ[X]_{\leq d-2}$.
\end{Lem}

\begin{proof}
    If $R$ is constant we have $|R(\xi)| = \normH{R}$ and the result is trivial. Now, suppose that $R$ is irreducible and not constant. We adapt the arguments of Davenport and Schmidt \cite[§5--6]{davenport1969approximation}. Set $H = e^{-n}\normH{R}$. By definition of $\homega_n(\xi)$ and $\homega$, if $H$ is sufficiently large, there exists a non-zero $P\in\bZ[X]_{\leq n}$ such that
    \[
        \normH{P} \leq H \AND |P(\xi)| \leq H^{-\homega}.
    \]
    By \eqref{eq: Gelfond's Lemma}, the (irreducible) polynomial $R$ is not a factor of $P$, they are thus coprime polynomials. Their resultant is a non-zero integer, and using Lemma~\ref{lem: estimation classique du resultant}, we obtain
    \begin{align*}
        1 & \ll \normH{P}^{\deg(R)-1}\normH{R}^{n}|P(\xi)| + \normH{P}^{\deg(R)}\normH{R}^{n-1}|R(\xi)| \\
          & \ll H^{n+\deg(R)-1-\homega} + H^{n+\deg(R)-1}|R(\xi)|.
    \end{align*}
    Since $\homega > 2n-d$ and $\deg(R) \leq n-d+1$, the first term tends to $0$ as $H$ tends to infinity. Hence $ 1 \ll H^{n+\deg(R)-1}|R(\xi)|$, which implies \eqref{lemme; eq1 : omega(R) petit degree}.

    \medskip

    If $R$ is not irreducible, we write $R = \prod_{i=1}^{s}R_i$ with integer $s\geq 1$ and $R_1,\dots,R_s\in\bZ[X]$ irreducible of degree $\leq \deg(R)$ (possibly constant). Combining $\normH{R} \asymp \prod_{i=1}^{s}\normH{R_i}$ together with  \eqref{lemme; eq1 : omega(R) petit degree} applied with the irreducible polynomials $R_i$, we find
    \begin{align*}
        |R(\xi)| = \prod_{i=1}^{s}|R_i(\xi)| & \gg \prod_{i=1}^{s}\normH{R_i}^{-(n + \deg(R) - 1)} \gg \normH{R}^{-(n + \deg(R) - 1)}.
    \end{align*}
    Finally, the last assertion comes from the fact that $d-1\leq n+d-1$ (since $d\leq 1+n/2$).
\end{proof}

\begin{Lem}
    \label{lem: existence facteur irreductible grand}
    Suppose that \eqref{eq: notation homega} holds. There exist $i_0\geq 0$ and a constant $c>0$ such that for each $i\geq i_0$ the polynomial $P_i$ has a unique irreducible factor $\tP_i\in\bZ[X]$ of degree $\geq n-d+2$ and positive leading coefficient. It satisfies
    \begin{equation}
        \label{eq lem: existence facteur irr grand eq 2}
        |P_i(\xi)| \normH{P_i}^{n+d-3} \geq c|\tP_i(\xi)| \normH{\tP_i}^{n+d-3},
    \end{equation}
    moreover $\big(\normH{\tP_i}\big)_{i\geq i_0}$ tends to infinity and as $i$ tends to infinity. For each $i$ large enough we have $\normH{\tP_i}>1$, and writing $|\tP_i(\xi)| = \normH{\tP_i}^{-\omega(\tP_i)}$, we furthermore have
    \begin{equation}
        \label{eq lem: existence facteur irr grand eq 1}
        \omega_n(\xi) = \limsup_{i\rightarrow\infty} \omega(\tP_i) \AND \liminf_{i\rightarrow\infty} \omega(\tP_i)  \geq \homega_n(\xi).
    \end{equation}
\end{Lem}

\begin{proof}
    First, note that since $d < 1+n/2$, if we decompose $P_i$ as a product of irreducibles, there is at most one factor of degree $\geq n-d+2$. Fix $i\geq 0$ large enough so that $\omega(P_i) \geq \homega$, and write
     \[
        P:=P_i = \prod_{k=1}^s R_k
     \]
     where $R_1,\dots,R_s\in\bZ[X]$ are irreducible polynomials (and $s$ is a positive integer). Suppose that $\deg(R_k) \leq n-d+1$ for each $k=1,\dots,s$. Then, by Lemma~\ref{lemme: omega(R) petit degre} together with $\normH{P} \asymp \prod_k \normH{R_k}$, we find
    \begin{align*}
        \normH{P}^{-\homega}\geq |P(\xi)| = \prod_{k=1}^{s} |R_k(\xi)| \gg \prod_{k=1}^{s} \normH{R_k}^{-(2n-d)} \asymp \normH{P}^{-(2n-d)}.
    \end{align*}
    This is impossible if $i$ is sufficiently large since $\homega > 2n-d$. Therefore, if $i$ is large enough, one of the factors $R_k$ has degree at least $n-d+2$. Without loss of generality, we may suppose that it is $R:=R_1$.
    Write $S:= \prod_{k = 2}^{s} R_k$, so that $P=RS$. We have $\deg(S) \leq d-2$, and \eqref{lemme; eq1 : omega(R) petit degree} of Lemma~\ref{lemme: omega(R) petit degre} yields 
    \begin{align*}
        |S(\xi)| \gg  \normH{S}^{-(n+d-3)}.
    \end{align*}
    Together with $\normH{P}\asymp \normH{R}\normH{S}$, it leads to
    \begin{align*}
        |P(\xi)| = |R(\xi)||S(\xi)| \gg |R(\xi)|\normH{S}^{-(n+d-3)} \asymp |R(\xi)|\normH{R}^{n+d-3}\normH{P}^{-(n+d-3)},
    \end{align*}
    and \eqref{eq lem: existence facteur irr grand eq 2} follows easily by setting $\tP_i:=R$. The rest of the proof is based solely on \eqref{eq lem: existence facteur irr grand eq 2} and the inequality $\normH{\tP_i} \ll \normH{P_i}$. Note that $|P_i(\xi)|\normH{P_i}^{n+d-3} \ll \normH{P_i}^{n+d-3-\homega}$ tends to $0$ as $i$ tends to infinity (using $d < 1+n/2$ together with $\omega(P_i) > 2n-d$). We deduce that $|\tP_i(\xi)|\normH{\tP_i}^{n+d-3}$ also tends to $0$ as $i$ tends to infinity, which is possible only if $\normH{\tP_i}$ tends to infinity. In particular, if $i$ is large enough we must have $\normH{\tP_i} > 1$. Writing $|\tP_i(\xi)| = \normH{\tP_i}^{-\omega(\tP_i)}$, we also have $\omega(\tP_i) > n+d-3$. Now, using $\normH{\tP_i}\ll \normH{P_i}$, and taking the $\log$ of the two sides of \eqref{eq lem: existence facteur irr grand eq 2}, we get
    \begin{align*}
        \big(\omega(P_i)-(n+d-3)\big)\log \normH{P_i} & \leq \big(\omega(\tP_i)-(n+d-3)\big)\log \normH{\tP_i} + \GrO(1) \\
        & \leq \big(\omega(\tP_i)-(n+d-3)\big)\big(\log \normH{P_i}+\GrO(1)\big) + \GrO(1).
    \end{align*}
    By dividing by $\log \normH{P_i}$ and by simplifying, we deduce that $\omega(\tP_i) \geq \omega(P_i)\big (1 - o(1)\big)$ and \eqref{eq lem: existence facteur irr grand eq 1} follows easily from \eqref{eq: encadrement omega(P_i)}.
\end{proof}

\begin{proof}[Proof of Proposition~\ref{prop: existence Q_i}]
    Let $i_0\geq 0$ and $(\tP_i)_{i\geq i_0}$ given by Lemma~\ref{lem: existence facteur irreductible grand}. Let $(Q_i)_{i\geq 0}$ be the (infinite) sequence of factors $(\tP_j)_{j\geq i_0}$ reordered by increasing height, without repetition. By Lemma~\ref{lem: existence facteur irreductible grand}, we may assume $i_0$ large enough so that $\normH{Q_i} > 1$ for each $i$, as well as $|Q_i(\xi)| \leq \normH{Q_i}^{-\homega}$. This sequence clearly satisfies the first assertions \ref{enum: property (Q_i) item 1} to \ref{enum: property (Q_i) item 6}, the third one coming  \eqref{eq lem: existence facteur irr grand eq 1} together with \eqref{eq: encadrement omega(P_i)}.

    \medskip

    Now, let $i\geq 0$ and let $j\geq i_0$ be an index such that $Q_i$ divides $P_j$. Since $\normH{Q_i}\ll \normH{P_j}$ by Gelfond's Lemma, the index $j$ tends to infinity as $i$ tends to infinity. Then, estimate \eqref{eq lem: existence facteur irr grand eq 2} can be rewritten as
    \begin{align}
        \label{eq proof: inter 1 pour controler H(P_j)}
         |P_j(\xi)|^{-1}\normH{P_j}^{-n-d+3} \ll |Q_i(\xi)|^{-1}\normH{Q_i}^{-n-d+3} = \normH{Q_i}^{\omega(Q_i)-n-d+3}.
    \end{align}
    Using $|P_j(\xi)|^{-1} \gg \normH{P_j}^{\homega}$ and $\homega > 2n-d$, we get, for each large enough $i$,
    \begin{align*}
        \normH{P_j}^{n-2d+3} \leq \normH{Q_i}^{\omega(Q_i)-n-d+3},
    \end{align*}
    which is equivalent to \eqref{eq: premiere estimation Pj vs Qi}. So, assertion~\ref{enum: H(P_j) controle par H(Q_i)} holds assuming $i_0$ large enough. 

    \medskip

    It remains to prove assertion~\ref{enum: H(Q_i+1) controle par H(Q_i)}. Note that this is trivial if $\omega_n(\xi) = \infty$. Let us assume that $\omega_n(\xi) < \infty$ and fix a small $\ee > 0$ to be chosen later. For each pair $(i,j)$ as above with $j \geq i_0$ large enough as a function of $\ee$, we have $\omega(P_j) > \homega_n(\xi)-\ee/2$ and $\omega(Q_i) < \omega_n(\xi)+\ee/2$, and thus \eqref{eq proof: inter 1 pour controler H(P_j)} yields
    \begin{align*}
        \normH{P_j}^{\homega_n(\xi)-\ee-n-d+3} \leq \normH{Q_i}^{\omega_n(\xi)+\ee-n-d+3},
    \end{align*}
    for each $i \geq 0$ and each $j\geq i_0$ such that $Q_i$ divides $P_j$. We define $k$ as the largest index such that
    \begin{align*}
        \normH{P_k} \leq \normH{Q_i}^{\theta(\ee)},\quad \textrm{where } \theta(\ee) = \frac{\omega_n(\xi)+\ee-n-d+3}{\homega_n(\xi)-\ee-n-d+3}.
    \end{align*}
    Since $\normH{P_j} \leq \normH{Q_i}^{\theta(\ee)}$, by maximality of $k$ we have $i_0\leq j\leq k$. Let $\ell$ be such that $Q_\ell$ divides $P_{k+1}$. We find
    \begin{align*}
        \normH{P_k} \leq \normH{Q_i}^{\theta(\ee)} < \normH{P_{k+1}} \leq \normH{Q_\ell}^{\theta(\ee)},
    \end{align*}
    and therefore $\ell \geq i+1$. On the other hand, since by Gelfond's Lemma we have $\normH{Q_\ell} \ll \normH{P_{k+1}}$, we deduce from \eqref{eq: H(P_i+1) controle par H(P_i)} that
    \begin{align*}
        \normH{Q_{i+1}} \leq \normH{Q_\ell} \ll \normH{P_{k+1}} \ll \normH{P_k}^{\omega_n(\xi)/\homega}\leq \normH{Q_i}^{\omega_n(\xi)\theta(\ee) /\homega}.
    \end{align*}
    We now choose $\ee>0$ small enough so that
    \begin{equation*}
        \theta(\ee) < \frac{\omega_n(\xi)-n-d+3}{\homega -n-d+3}.
    \end{equation*}
    This is possible since $\homega < \homega_n(\xi)$, and it yields \eqref{eq: estimation Q_i+1 par Q_i} for each $i\geq 0$, assuming that $i_0$ is large enough.
\end{proof}

\section{On the dimension of some polynomials subspaces}
\label{section: espaces V_N}

We start by introducing some families of vector spaces spanned by polynomials, and we study their dimensions.

\begin{Def}
    \label{Def: fonction g_A(N)}
    Let $k\geq n$ be an integer and let $\cA$ be a subset of $\bR[X]_{\leq n}$. We define
    \begin{align*}
        \cB_k(\cA) &= \big\{ Q, XQ,\dots, X^{k-\deg(Q)}Q \,;\, Q\in\cA\setminus\{0\} \big\} \subset \bR[X]_{\leq k}, \\
        V_k(\cA) &= \Vect[\bR]{\cB_k(\cA)}, \\
        g_\cA(k) &= \dim V_k(\cA).
    \end{align*}
\end{Def}

The spaces $V_k(\cA)$ play the role of the spaces $\cU^{k}(\cA)$ in \cite[Section~3]{poels2022simultaneous} (for simultaneous approximation). We obtain analog properties. Note that if $\cA$ contains at least one non-zero polynomial, then
\begin{equation}
    \label{eq: suite V_k(A) strict. croissante}
    V_n(\cA) \varsubsetneq V_{n+1}(\cA) \varsubsetneq\cdots.
\end{equation}

The goal of this section is to prove the following result. We could not find a reference for the proposition below.

\begin{Prop}
    \label{Cor: V_2n-k =  tout l'espace}
    Let $k$ be an integer with $0\leq k \leq n$, and let $\cA$ be a set of $k+1$ linearly polynomials of $\bR[X]_{\leq n}$. Suppose that the $\gcd$ of the elements of $\cA$ is $1$ (in other words, the ideal spanned by $\cA$ is $\bR[X]$). Then
    \begin{equation}
        \label{Cor: V_N = tout l'espace}
        V_{2n-k}(\cA) = \bR[X]_{\leq 2n-k}.
    \end{equation}
\end{Prop}

The case $k=1$ is a classical result (it is implied by the fact that the resultant of two coprime polynomials is non-zero). The proof of Proposition~\ref{Cor: V_2n-k =  tout l'espace} is given at the end of the section. Recall that a function  $f : \{n,n+1, \ldots \} \rightarrow \bR$ is \textsl{concave} if for any $i > n$, it satisfies
\[
    f(i)-f(i-1) \geq f(i+1)-f(i).
\]
The next result is a dual version of \cite[Proposition 3.1]{poels2022simultaneous} (where we deal with simultaneous approximation to the successive powers of $\xi$).

\begin{Lem}
   Let $\cA \neq \{0\}$ be a non-empty subset of $\bR[X]_{\leq n}$. The function $g_\cA$ is concave and (strictly) increasing on $\{n,n+1,\ldots\}$.
\end{Lem}

\begin{proof}
    The series of inclusion \eqref{eq: suite V_k(A) strict. croissante} shows that the function $g_\cA$ is increasing on $\{n,n+1,\dots\}$. For simplicity, we write $V_i=V_{i}(\cA)$ and $\cB_{i} = \cB_{i}(\cA)$ for each $i \geq n$. Given an integer $i\geq n$ we have $XV_{i}\subset V_{i+1}$, and we set
    \[
        h(i):= \dim \big(V_{i+1}/XV_{i}\big) = g_\cA(i+1)-g_\cA(i).
    \]
    We have to prove that $h$ is decreasing on $\{n,n+1,\cdots\}$. Fix $i\geq n+1$ and consider the linear map $\pi : V_i \rightarrow V_{i+1}/XV_{i}$ defined by $\pi(P) = P + XV_{i}$. Since $\cB_{i}\cup X\cB_{i} = \cB_{i+1}$, we have $V_i + XV_i = V_{i+1}$. So $\pi$ is surjective, and consequently $\textrm{Im } \pi = V_{i+1}/XV_{i}$ is isomorphic to $V_i / \ker \pi$.
    On the other hand, $XV_{i-1}\subset V_{i} \cap XV_{i} \subset \ker\pi$, so $XV_{i-1}$ is  subspace of $\ker \pi$.  Hence
    \begin{align*}
       h(i-1) = \dim \big(V_{i}/XV_{i-1}\big)  \geq \dim\big(V_i / \ker \pi \big) = \dim \big(V_{i+1}/XV_{i}\big) = h(i).
    \end{align*}
\end{proof}

\begin{Lem}
    \label{lem : dim V_n+j(P,Q)}
    Let $P,Q \in \bR[X]_{\leq n}$ be two coprime polynomials. Then, we have
    \begin{align*}
        \dim V_{n+j}(P,Q) \geq 2(j+1),
    \end{align*}
    for each $j\in\{0,\dots,n-1\}$. In particular $V_{2n-1}(P,Q) = \bR[X]_{\leq 2n-1}$.
\end{Lem}

\begin{proof}
    Let $p$ (resp. $q$) denote the degree of $P$ (resp. of $Q$). There exist $\alpha,\beta\in\bR$ such that the polynomial $\tP := P(X)(X-\alpha)^{n-p}$ and $\tQ:= Q(X)(X-\beta)^{n-q}$ are coprime (and of degree exactly $n$). Fix $j\in\{0,\dots,n-1\}$. Since $\tP$ and $\tQ$ are coprime and $j<n$, the linear map
    \begin{equation*}
        \left\{ \begin{array}{ccc}
                  \bR[X]_{\leq j}\times \bR[X]_{\leq j}  & \longrightarrow &  \bR[X]_{\leq n+j} \\
                     (R,S)     & \longmapsto & R\tP + S\tQ
                \end{array}\right.
    \end{equation*}
    is injective, so its image $V_{n+j}(\tP,\tQ) \subset V_{n+j}(P,Q)$ has dimension $2(j+1)$.
\end{proof}

\begin{proof}[Proof of Proposition~\ref{Cor: V_2n-k =  tout l'espace}]
     For simplicity we write $g=g_\cA$. Recall that $\cA$ has cardinality $k+1$, so that $g(n)\geq \textrm{card}(\cA) = k+1$. If $k=n$, then \eqref{Cor: V_N = tout l'espace} is automatic (since in that case $\cA$ contains a basis of $\bR[X]_{\leq n}$). So, we may assume that $k< n$. We first prove that for each sufficiently large $m$, we have
    \begin{align}
        \label{eq : il existe m tel que V_m = tout}
            V_{m}(\cA) = \bR[X]_{\leq m}.
    \end{align}
    Indeed, since the ideal spanned by $\cA$ is $\bR[X]$, there exists an integer $\ell\geq n$ such that $1\in V_\ell(\cA)$. Let $P$ be a non-zero element in $\cA$ of degree $d$, and set $m=\ell+d$. Then $V_{m}(\cA)$ contains $\bR[X]_{\leq d}$, as well as the polynomials $P,XP,\cdots,X^\ell P$. We easily deduce \eqref{eq : il existe m tel que V_m = tout}.

    \medskip

    By contradiction, suppose that \eqref{Cor: V_N = tout l'espace} does not hold, \ie
    \begin{align}
        \label{eq proof:Cor: V_N = tout l'espace: by contradiction}
        g(2n-k) \leq 2n-k.
    \end{align}
    We distinguish between two cases. Suppose first that $g(2n-k)-g(2n-k-1) \geq 2$. By concavity, then $g(j)-g(j-1)\geq 2$ for each $j$ with $n < j \leq n-k$, and we deduce that
    \begin{align*}
        g(2n-k) \geq g(n)+2(n-k) \geq k+1 + 2(n-k) = 2n-k+1,
    \end{align*}
    since $g(n) \geq \textrm{card}(\cA) = k+1$. This contradicts \eqref{eq proof:Cor: V_N = tout l'espace: by contradiction}, so $g(2n-k)-g(2n-k-1) \leq 1$. Since the function $g$ is increasing and concave, it is linear with slope $1$ on $\{2n-k,2n-k+1,\dots\}$. Choosing $m > 2n-k$ such that \eqref{eq : il existe m tel que V_m = tout} holds, we obtain by \eqref{eq proof:Cor: V_N = tout l'espace: by contradiction}
    \begin{align*}
        m+1 = g(m) = g(2n-k)+ m - (2n-k) \leq m,
    \end{align*}
    a contradiction. Hence $g(2n-k) = 2n-k+1$, or equivalently, \eqref{Cor: V_N = tout l'espace} holds.
\end{proof}

\begin{figure}[H]
    \begin{tikzpicture}[xscale=0.5,yscale=0.4, line cap=round,line join=round,>=triangle 45,x=1cm,y=1cm]
        \clip(-5,-2) rectangle (20,15);
        \draw[-stealth, semithick] (-0.15,0)--(18,0) node[below]{$i$};
        \draw[-stealth, semithick] (0,-0.15)--(0,14.5) node[left]{$g(i)$};

        \draw [thick] (0,1)-- (1,5);
        \draw [thick] (1,5)-- (2.66,8.059);
        \draw [thick] (2.66,8.059)-- (7,10.5);
        \draw [thick] (7,10.5)-- (14,12);
        \draw [dashed] (7,13.34)-- (7,-0.5);

        \node[left] at (-0,1.7) {$\geq \textrm{card}(A)$};
        \node[above] at (4,12) {slope $\geq 2$};
        \node[above] at (10,12) {slope $= 1$};
        \node[below] at (7,0)  {$2n-\ell$};
        \node[below] at (14,0)  {$2n-1$};
        \node[left] at (0,13)  {$2n$};

        \draw [dash pattern=on 2pt off 3pt](0,12)-- (14,12);
        \draw [dash pattern=on 2pt off 3pt] (14,12)-- (14,0);

       \node[below] at (0,-0.15) {$i=n$};
       \node[draw,circle,inner sep=1.25pt,fill] at (0,1) {};
       \node[draw,circle,inner sep=1.25pt,fill] at (1,5) {};
       \node[draw,circle,inner sep=1.25pt,fill] at (2.666,8.059) {};
       \node[draw,circle,inner sep=1.25pt,fill] at (7,10.5) {};
       \node[draw,circle,inner sep=1.25pt,fill] at (14,12) {};
       \node[draw,circle,inner sep=1.25pt,fill] at (0,12) {};
       \node[draw,circle,inner sep=1.25pt,fill] at (14,0) {};
       \node[draw,circle,inner sep=1.25pt,fill] at (7,0) {};
       \node[draw,circle,inner sep=1.25pt,fill] at (0,0) {};
    \end{tikzpicture}
    \caption{Graph of the piecewise linear function interpolating the
    values $g(i)=\dim V_i(\cA)$ at integers $i\in\{n,\dots,2n-1\}$.}
    \label{fig:1}
\end{figure}
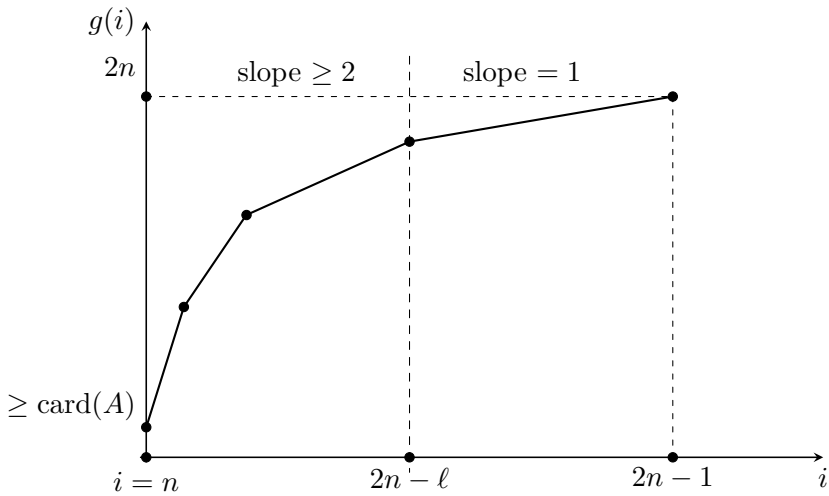

\section{Proof of Theorem~\ref{Thm : main d=2} (case $d=2$)}
\label{Section: cas d = 2}

In this section, we deal with the case $d=2$ to prove Theorem~\ref{Thm : main d=2}, namely that $\homega_3(\xi)\leq 2+\sqrt 5  = 4.23\cdots$ and  $\homega_n(\xi) \leq 2n-2$ for each $n\geq 4$. This was already known for $n\geq 10$, however for $n=3,\dots,9$ it is a new result. For $n=3$, our bound improves on the bound $\homega_3(\xi)\leq 3+\sqrt 2 = 4.41\cdots$ due to Bugeaud and Schleischitz \cite{bugeaudSchleischitz2016Uniform}. Moreover, our proof does not require Marnat-Moshchevitin's inequality \cite{marnat2018optimal}.

\begin{proof}[Proof of Theorem~\ref{Thm : main d=2}]
    Suppose that $\homega_n(\xi) > 2n-2$, and fix a real number $\homega$ such that
    \[
        \homega_n(\xi) > \homega > 2n-2.
    \]
    Let $(P_i)_{i\geq 0}$ be a sequence of minimal polynomials associated to $n$ and $\xi$ as in Section~\ref{section: minimal pol}. According to Lemma~\ref{lem: existence facteur irreductible grand} (with $d=2$) there exists an index $i_0\geq 0$ such that $P_i$ has degree $n$ and is irreducible for each $i\geq i_0$. Consequently, up to a finite number of terms, the sequence $(P_i)_{i\geq 0}$ coincides with the sequence $(Q_i)_{i\geq 0}$ of Proposition~\ref{prop: existence Q_i}. Let $I$ denotes the set of indices $i\geq i_0+1$ such that $P_{i-1}$, $P_i$ and $P_{i+1}$ are linearly independent. By Lemmas~\ref{Lem: P_iH_i+1 = P_j-1H_j} and~\ref{lem: I infini}, the set $I$ is infinite, and for any consecutive $i<j$ in $I$, we have
    \[
        \normH{P_{i+1}}|P_i(\xi)| \asymp \normH{P_j}|P_{j-1}(\xi)|.
    \]
    Furthermore, the irreducible polynomials $P_i$ and $P_{i+1}$ are also coprime since $\normH{P_i} < \normH{P_{i+1}}$ and $|P_i(\xi)| < |P_{i+1}(\xi)|$. Lemma~\ref{lem: estimation classique du resultant} yields
    \begin{align*}
        1 \leq |\Res(P_i,P_{i+1})| \ll \normH{P_i}^{n-1}\normH{P_{i+1}}^n|P_{i}(\xi)| &\ll \normH{P_i}^{n-1}\normH{P_{j}}^n|P_{j-1}(\xi)| \\
        & \ll \normH{P_i}^{n-1}\normH{P_{j}}^{n-\homega}.
    \end{align*}
    We deduce that
    \begin{align}
        \label{eq proof: estimation theta}
        \normH{P_{j}} \leq \normH{P_i}^\theta \quad \textrm{where } \theta = \frac{n-1}{\homega-n}.
    \end{align}
    Let $h<i<j$ be consecutive indices in $I$. We have the following configuration
    \[
        \Vect[\bR]{P_h,P_{h+1}} = \Vect[\bR]{P_{i-1},P_{i}} \neq \Vect[\bR]{P_{i},P_{i+1}},
    \]
    so $P_{h}, P_{h+1}, P_{i+1}$ are linearly independent. Proposition~\ref{Cor: V_2n-k =  tout l'espace} combined with Lemma~\ref{lem : dim V_n+j(P,Q)} implies that
    \begin{align*}
        \Big(\bR[X]_{\leq n-2}P_{h}\oplus \bR[X]_{\leq n-2}P_{h+1}\Big) + \bR[X]_{\leq n-2}P_{i+1} = \bR[X]_{\leq 2n-2}.
    \end{align*}
     Choose $k\in\{0,\dots,n-2\}$ such that $\big(P_{h},\dots,X^{n-2}P_{h},P_{h+1},\dots,X^{n-2}P_{h+1}, X^kP_{i+1}\big)$ is a basis of $\bR[X]_{\leq 2n-2}$. We denote by $M$ the matrix of this basis expressed into the canonical basis $(1,X,\dots,X^{2n-2})$.  Estimating $\det(M)$ as in the proof of Lemma~\ref{lem: estimation classique du resultant} (in other words, for $\ell=2,\dots,2n-2$, we add to the first row of $M$  the $\ell$-th row multiplied by $\xi^{\ell-1}$), we get the estimates
    \begin{align*}
        1  \leq \big|\det\big(M)\big| & \ll |P_{h}(\xi)|\normH{P_{h}}^{n-2}\normH{P_{h+1}}^{n-1}\normH{P_{i+1}}.
    \end{align*}
    Now, since $\normH{P_{h+1}}^{n-1}|P_{h}(\xi)| \asymp \normH{P_{h+1}}^{n-2}\normH{P_i}|P_{i-1}(\xi)| \ll \normH{P_i}^{n-1-\homega}$, we deduce that
    \begin{align}
        \label{eq proof: estimation tau}
        \normH{P_i}^{\homega-n+1} \ll \normH{P_{h}}^{n-2}\normH{P_{j}}.
    \end{align}

    For each consecutive $i<j$ in $I$, define $\tau_i\in(0,1)$ by
    \[
        \normH{P_i} = \normH{P_j}^{\tau_i}
    \]
    and set $\tau = \limsup_{i\in I, i\rightarrow\infty} \tau_i \in [0,1]$. Let $h<i<j$ be consecutive indices in $I$ as previously. By \eqref{eq proof: estimation tau}, we obtain
    \[
        \homega-n+1  \leq (n-2)\tau_h + \frac{1}{\tau_i} +o(1) \leq (n-2)\tau + \frac{1}{\tau_i} +o(1).
    \]
    We infer that
    \begin{align}
        \label{eq proof: majoration homega_3: eq 2}
        p(\tau) \geq 0, \quad \textrm{where } p(t) = (n-2)t^2-(\homega-n+1)t+1.
    \end{align}
    Note that
    \begin{align*}
        p(0) &= 1, \\
        p\left(\frac{1}{n-2}\right) & = \frac{2n-2-\homega}{n-2} < 0, \\
        p(1) &= 2n-2-\homega < 0.
    \end{align*}
    We deduce that $p$ has one root $\alpha\in(0,1/(n-2))$ and one root larger than $1$. Since $\tau\in[0,1]$ and $p(\tau)\geq 0$, we obtain $\tau \leq \alpha$. Combined with the estimate $\normH{P_i} = \normH{P_j}^{\tau_i}  \ll \normH{P_i}^{\theta\tau_i}$ valid
    for any $i\in I$ (it is a consequence of \eqref{eq proof: estimation theta}), this leads to
    \begin{align}
        \label{eq proof: majoration homega_3: eq 3}
        1 \leq \theta\tau \leq \theta\alpha < \frac{n-1}{(n-2)^2}.
    \end{align}
    We easily check that it is impossible when $n\geq 4$ (the right-hand side is strictly less than $1$), so that $\homega_n(\xi)\leq 2n-2$ for each $n\geq 4$.

    \medskip

    We now deal with the case $n=3$. Suppose by contradiction that $\homega_3(\xi) > 2+\sqrt 5$ and choose $\homega$ such that
    \[
        \homega_3(\xi) > \homega > 2+\sqrt 5.
    \]
    The polynomial $p$ defined in \eqref{eq proof: majoration homega_3: eq 2} satisfies $p(t) = t^2-(\homega-2)t+1$. Denote by $\alpha$ its smallest root, and by $\beta = (\sqrt 5 -1)/2$ be the smallest one of the polynomial $t^2-\sqrt 5 t + 1$. We find
    \begin{align*}
        0 = \beta^2-\sqrt 5 \beta + 1 > \beta^2-(\homega-2 )\beta + 1 = p(\beta),
    \end{align*}
    hence $\alpha < \beta$. Combined with $\theta = 2/(\homega-3) < 1/\beta$, it implies that $\theta \alpha < 1$, which contradicts \eqref{eq proof: majoration homega_3: eq 3}. It follows that $\homega_3(\xi) \leq 2+\sqrt 5$.
\end{proof}


\section{Multilinear algebra and height of polynomial subspaces}

This section is divided into two parts. We introduce and study a quantity $\HHstar(V)$ associated to a subspace $V\subset\bR^m$ defined over $\bQ$ in Section~\ref{def petitesse HHstar(V)}. Intuitively, $\HHstar(V)$ is small if $V$ is spanned by good polynomials approximations of $\bZ[X]$ (which are small when evaluated at $\xi$). This will be a key-point for estimating the height of the polynomials $Q_i$ of Section~\ref{section: sequence des Q_i}. In order to define $\HHstar$, we need some tools of multilinear algebra that we recall in Section~\ref{section: hauteurs tordues}. In Appendix~\ref{subsection: twisted heights} we give another interpretation of $\HHstar$ in term of twisted heights.

\subsection{Multilinear algebra and Hodge duality}
\label{section: hauteurs tordues}

For each integer $m$, we view $\bR^{m+1}$ as an Euclidean space for the usual scalar product $\psc{\cdot}{\cdot}$, and we denote by $\norm{\cdot}$ the associated Euclidean norm. For each  $k = 1, \cdots , m+1$, we identify $\bigwedge^k \bR^{m+1}$ with $\bR^N$, where $N = \binom{m+1}{k}$, via a choice of ordering of the Plücker coordinates, and we denote by $\norm{\by}$ the norm of a point $\by \in \bigwedge^k \bR^{m+1}\cong \bR^N$. This is independent of the ordering of its coordinates. Let $V$ be a $k$-dimensional subspace of $\bR^{m+1}$ defined over $\bQ$, \ie such that $\Vect[\bR]{V\cap\bQ^{m+1}} = V$. Its (standard) \textsl{height} $\HH(V)$ is the covolume of the lattice $V\cap\bZ^{m+1}$ inside $V$ (with the convention that $\HH(V) = 1$ if $V=\{0\}$). Explicitly, we have
\[
    \HH(V) := \norm{ \bx_1\wedge \cdots \wedge \bx_k},
\]
for any $\bZ$--basis $(\bx_1,\dots,\bx_k)$ of the lattice $V\cap \bZ^{m+1}$. Schmidt established the very nice inequality
\begin{equation*}
    \HH(U\cap V)\HH(U+V) \leq \HH(U)\HH(V),
\end{equation*}
valid for any subspaces $U,V$ of $\bR^{m+1}$ defined over $\bQ$ (see \cite[26, Chapter I, Lemma 8A]{Schmidt}). In this paper, we need to work with a ``twisted'' height and the corresponding version of Schmidt's inequality (which is obtained by following Schmidt's original arguments).

\medskip

Let $(\be_1,\dots,\be_{m+1})$ denote the canonical basis of $\bR^{m+1}$, and let $k$ be an integer with $0\leq k \leq m+1$. The Hodge star operator
\begin{align*}
    * : {\bigwedge}^k(\bR^{m+1}) \mathop{\longrightarrow}^{\sim} {\bigwedge}^{m+1-k}(\bR^{m+1})
\end{align*}
is defined by
\begin{align*}
    *(\be_{i_1}\wedge \cdots \wedge \be_{i_k}) = \ee_{i_1,\dots,i_k}\be_{j_1}\,\wedge \cdots \wedge \be_{j_{m+1-k}}
\end{align*}
for any indices $i_1 < \cdots < i_k$ and $j_1 < \cdots < j_{m+1-k}$ forming a partition of $\{1,\dots,m+1\}$, where $\ee_{i_1,\dots,i_k}$ denotes the signature of the substitution $(1,\dots,m+1) \mapsto (j_1,\dots,j_{m+1-k}, i_1,\dots,i_k)$. Given $\bX\in \bigwedge^k\bR^{m+1}$, the point $*\bX$ is called the \textsl{Hodge dual} of $\bX$.

\medskip

We now collect some useful properties of the Hodge star operator, see for example \cite{Hodge}, \cite{bourbakiAlgebre} and \cite[Section 3]{bugeaud2010transfer} for more details. First,
\[
    \norm{*\bX} = \norm{\bX} \AND *(*\bX) = (-1)^{k(m+1-k)}\bX
\]
for any $\bX\in \bigwedge^k\bR^{m+1}$. If $\bX = \bx_1\wedge \cdots \wedge \bx_k$ is a system of  Plücker coordinates of a $k$-dimensional subspace $V\subset \bR^{m+1}$, then $*\bX$ is a system of Plücker coordinates of its orthogonal $V^\perp$. This implies the classical identity
\[
    \HH(V) = \HH(V^\perp).
\]
If $k>0$, then given $\by\in\bR^{m+1}$ and a multivector $\bX\in \bigwedge^k\bR^{m+1}$, the point
\begin{align*}
    \by \lrcorner\, \bX = *\big(\by\wedge (*\bX)\big) \in {\bigwedge}^{k-1}(\bR^{m+1})
\end{align*}
is called the \textsl{contraction} of $\bX$ by $\by$ (see \cite[Lemma 2]{bugeaud2010transfer}). Explicitly, if $\bX = \bx_1\wedge \cdots \wedge \bx_k$ is a decomposable multivector, then
\begin{align}
    \label{eq: contraction formule explicite}
    \by \lrcorner\, \bX = \sum_{i=1}^{k} (-1)^{k-i} \psc{\bx_i}{\by} \, \bx_1\wedge \cdots \wedge \widehat{\bx_i} \wedge \cdots \wedge \bx_k,
    \end{align}
where the hat on $\bx_i$ means that this term is omitted  from the wedge product (see \cite[Eq (3.3)]{bugeaud2010transfer}). In particular, if $k=1$ and $\bX = \bx\in\bR^{m+1}$, we simply have
\begin{align}
    \label{eq: contraction et produit scalaire}
    \by \lrcorner\, \bx =  \psc{\by}{\bx}.
\end{align}

\subsection{Schmidt's inequality}
\label{def petitesse HHstar(V)}

Let $m$ be a non-negative integer and set $\Xi_m = (1,\xi,\xi^2,\dots,\xi^m)$. We keep the notation of Section~\ref{section: hauteurs tordues}.

\begin{Def}
    \label{Def: def HHstar}
    Let $V$ be a $k$-dimensional subspace of $\bR^{m+1}$ defined over $\bQ$, with $k\geq 1$, and let $(\bx_1,\dots,\bx_k)$ be a $\bZ$--basis of the lattice $V\cap \bZ^{m+1}$. We set
    \begin{align*}
        \HHstar(V) = \norm{\Xi_m \, \lrcorner\, \bX } = \norm{\Xi_m \wedge (*\bX)},
    \end{align*}
    where $\bX = \bx_1\wedge \cdots \wedge\bx_k$. By convention, we set $\HHstar(\{0\}) = 0$. Following the notation of \cite[Section 11]{poels2022simultaneous}, we also set
    \begin{align*}
        L_\xi(V) = \norm{\Xi_m \wedge \bX},
    \end{align*}
    with the convention that $L_\xi(\{0\}) = \norm{\Xi_m}$.
\end{Def}

\noindent\textsl{Remark.} If $(\bx_1',\dots,\bx_k')$ is another $\bZ$--basis of $V\cap \bZ^{m+1}$, then $\bx_1'\wedge \cdots \wedge\bx_k' = \pm \bX$. Consequently, $\HHstar(V)$ and $L_\xi(V)$ do not depend on the choice of the basis. In \cite{poels2022simultaneous}, we considered $L_\xi(V)$ for spaces $V$ spanned by good simultaneous approximations. The function $\HHstar$ is connected to the quantity introduced in \cite[Definition~7.1]{poelsroy2022parametric} (where we work in a number $K$ instead of $\bQ$). Note that $\HHstar(V) = 0$ if and only if $\Xi_m\in V^\perp$. Since $\xi$ is transcendental, it is only possible when $V=\{0\}$. We have
\begin{align*}
    \HHstar(\bR^{m+1}) = \norm{\Xi_m } \asymp 1,
\end{align*}
where the implicit constants depend on $\xi$ and $m$ only. Moreover, \eqref{eq: contraction et produit scalaire} implies that
\begin{align}
    \label{eq: HHstar(bx)}
    \HHstar\big(\Vect[\bR]{\bx}\big) =  |\psc{\Xi_m}{\bx}|
\end{align}
for any primitive integer point $\bx\in\bZ^{m+1}$. Eq. \eqref{eq: contraction formule explicite} yields the explicit formula
\begin{align}
    \label{eq: HHstar(V) formule explicite}
    \HHstar(V) = \norm{ \sum_{i=1}^{k} (-1)^{k-i}\psc{\bx_i}{\Xi_m} \, \bx_1\wedge \cdots \wedge \widehat{\bx_i} \wedge \cdots \wedge \bx_k }.
\end{align}
On the other hand, if $(\by_1,\dots,\by_{m+1-k})$ is a $\bZ$--basis of  $V^\perp\cap \bZ^{m+1}$,  then $*\bX = \pm \by_1\wedge \cdots \wedge \by_{m+1-k}$. Consequently, we can also write
\begin{align}
    \label{eq: D_xi(V) = L_xi(V^perp)}
    \HHstar(V) = \norm{\Xi_m \wedge \by_1\wedge\cdots \wedge \by_{m+1-k}} = L_\xi(V^\perp).
\end{align}
Both expressions of $\HHstar(V) $ will be useful.

\begin{Prop}[Schmidt's inequality]
    \label{Prop: Schmidt inegalite pour HHstar}
    For any subspaces $U,V$ of $\bR^{m+1}$ defined over $\bQ$, we have
    \begin{equation}
        \label{eq: Schmidt's inequality HHstar}
        \HHstar(U\cap V)\HHstar(U+V) \leq \HHstar(U)\HHstar(V)
    \end{equation}
    and
    \begin{align}
        \label{eq: Schmidt's inequality L_xi}
        L_\xi(U \cap V) L_\xi(U + V) \leq L_\xi(U)L_\xi(V).
    \end{align}
\end{Prop}

\begin{proof}
    In view of \eqref{eq: D_xi(V) = L_xi(V^perp)}, we only need to prove that \eqref{eq: Schmidt's inequality L_xi} holds for any pair $(U,V)$ as in the statement of the proposition (then, it suffices to apply \eqref{eq: Schmidt's inequality L_xi} to the pair $(U^\perp,V^\perp)$). We follow Schmidt's arguments \cite[Chapter I, Lemma 8A]{Schmidt}. 
    For any $\bX, \bY, \bZZ \in \bigwedge \bR^{m+1}$ which are pure products of elements in $\bR^{m+1}$, we have
    \begin{align}
        \label{eq proof: prop Schmidt inegalite pour HHstar}
        \norm{\bX} \norm{\bX\wedge \bY \wedge \bZZ} \leq \norm{\bX\wedge \bY} \norm{\bX\wedge \bZZ}.
    \end{align}
    Let $U, V$ be subspaces of $\bR^{m+1}$ defined over $\bQ$. If $U=\{0\}$ or $V = \{0\}$, then \eqref{eq: Schmidt's inequality L_xi} is trivial, so we may assume that $U$ and $V$ have dimension $\geq 1$. Let $\bx_1,\dots,\bx_r$ be a $\bZ$--basis of $U\cap V \cap \bZ^{m+1}$, which we complete to a $\bZ$--basis $\bx_1,\dots,\bx_r,\by_1,\dots,\by_s$ of $U\cap\bZ^{m+1}$ (resp. $\bx_1,\dots,\bx_r,\bz_1,\dots,\bz_t$ of $V\cap\bZ^{m+1}$). Set
    \begin{align*}
        \bX = \Xi_m\wedge \bx_1 \wedge \dots \bx_r, \quad \bY = \by_1\wedge \dots \wedge \by_s \AND \bZZ = \bz_1 \wedge \cdots \wedge\bz_t.
    \end{align*}
    We get \eqref{eq: Schmidt's inequality L_xi} by applying \eqref{eq proof: prop Schmidt inegalite pour HHstar} with the above choice of pure products.
\end{proof}

We identify $\bR[X]_{\leq m}$ to $\bR^{m+1}$ and  $\bR^{m+1}$ to the space of $(m+1)\times 1$ columns matrices with real coefficients via the isomorphisms
\begin{equation}
    \label{eq: identification poly avec points de R^m}
    \sum_{k=0}^{m} a_k X^k \longmapsto (a_0,\dots,a_{m}) \AND
    (a_0,\dots,a_{m})\mapsto
    \left(\begin{array}{c}
      a_0 \\
      \vdots \\
      a_{m}
    \end{array} \right).
\end{equation}
Then, for any $P\in\bR[X]_{\leq m}$, we have $P(\xi) = \psc{\bz}{\Xi_m}$, where $\bz\in\bR^{m+1}$ corresponds to $P$. In particular, if $P\in\bZ[X]_{\leq m}$ is primitive, then \eqref{eq: HHstar(bx)} may be rewritten as
\begin{align}
    \label{eq: HHstar(P)}
    \HHstar\big(\Vect[\bR]{P}\big) =  |P(\xi)|.
\end{align}
We will repeatedly use the following ``twisted'' dual version of \cite[Lemma 2.1]{poels2022simultaneous}. Intuitively, it implies that if $V$ is spanned by polynomials $P$ such that $|P(\xi)|$ is small, then $\HHstar(V)$ is also small. This generalizes the inequality
\[
    \HHstar(\Vect[\bR]{P}) \leq |P(\xi)|
\]
valid for any $P\in\bZ[X]_{\leq m}$.

\begin{Lem}
    \label{lem: estimation des pseudo resultant tordus}
    There is a positive constant $c$, which only depends on $n$ and $\xi$, with the following property. For any linearly independent polynomials $P_1,\dots,P_k\in\bZ[X]_{\leq m}$ (with $k\geq 1$), we have
    \begin{align}
        \label{eq: estimation L_xi(V) avec T}
        \HHstar\big(\Vect[\bR]{P_1,\dots,P_k} \big) & \leq c \sum_{i=1}^{k} \frac{|P_i(\xi)|}{\normH{P_i}} \prod_{j=1}^{k} \normH{P_j}.
    \end{align}
\end{Lem}

\begin{proof}
    Let $Q_1,\dots,Q_k$ be a $\bZ$-- basis of $V\cap\bZ[X]_{\leq m}$, where  $V=\Vect[\bR]{P_1,\dots,P_k}$. There exists a non-zero $\alpha\in\bZ$ such that
    \[
        P_1\wedge \cdots \wedge P_k = \alpha Q_1\wedge \cdots \wedge Q_k,
    \]
    and so
    \[
        \HHstar(V) = \norm{ \Xi_m\,\lrcorner \,(Q_1\wedge \cdots \wedge Q_k)} \leq \norm{ \Xi_m\,\lrcorner \,(P_1\wedge \cdots \wedge P_k)}.
    \]
    On the other hand, by \eqref{eq: contraction formule explicite} combined with Hadamard's inequality, we obtain
    \begin{align*}
        \norm{\Xi_m\,\lrcorner \,(P_1\wedge \cdots \wedge P_k)} & = \norm{\sum_{i=1}^{k} (-1)^{k-i} P_i(\xi) \times P_1\wedge \cdots \wedge \widehat{P_i} \wedge \cdots \wedge P_k } \\
        & \ll \sum_{i=1}^{k} |P_i(\xi)| \normH{P_1}\cdots \widehat{\normH{P_i}} \cdots \normH{P_k}
    \end{align*}
    (recall that the naive height $\normH{\cdot}$ is defined in Section~\ref{section: notation}).
\end{proof}

\section{Subfamilies of polynomials: dimension and height}
\label{section : indpce lineaire Q_i}

Let $d,n,\xi$ and $\homega$ be as in Section~\ref{section: sequence des Q_i}. In particular we have
\[
    2\leq d < 1+\frac{n}{2},
\]
and we suppose that \eqref{eq: notation homega} holds, namely
\[
    \homega_n(\xi) > \homega > 2n -d.
\]
Let us fix a sequence of minimal polynomials $(P_i)_{i\geq 0}$ associated to $n$ and $\xi$ as in Section~\ref{section: minimal pol}. We denote by $(Q_i)_{i\geq 0}$ the sequence of irreducible factors given by Proposition~\ref{prop: existence Q_i}. In particular, for each $i\geq 0$ we have
\begin{equation}
    \label{eq: petitesse de Q_i(xi)}
    |Q_i(\xi)| \leq \normH{Q_i}^{-\homega},
\end{equation}
as well as
\begin{align}
    \label{eq: taille de H(Q_i+1)}
    \normH{Q_{i+1}}^\tau \leq \normH{Q_{i}}, \quad \textrm{where } \tau = \frac{\homega\big(\homega -n-d+3\big)}{\omega_n(\xi)\big(\omega_n(\xi)-n-d+3\big)} \in [0,1).
\end{align}
Under the hypothesis that $d$ is not too large, we will prove in the next section that $\omega_n(\xi) < \infty$, and thus $\tau > 0$.
Here, we investigate the following question: can we find ``large'' subfamilies of $(Q_i)_{i\geq 0}$ which are linearly independent and whose elements have ``comparable'' height? More precisely, given two indices $k< i$, can we find an exponent $\theta_j\in(0,1)$ which depends only on $d,n$ and the dimension  $j+1$ of the subspace $\Vect[\bR]{Q_k,Q_{k+1},\dots,Q_i}$ (and not on the indices $i$ and $k$), such that $\normH{Q_i}^{\theta_j} \ll \normH{Q_k}$? We can view it as a generalization of \eqref{eq: taille de H(Q_i+1)}. With this goal in mind, let us introduce some notation.

\begin{Def}
    \label{Def m_n}
    Let $m_n = m_n(\xi) \in [2,n+1]$ be the integer
    \[
        m_n := \lim_{i\rightarrow\infty} \dim(\Vect[\bR]{Q_i,Q_{i+1},\dots}).
    \]
\end{Def}

\noindent\textsl{Remark.} Note that we might have $m_n < n+1$, since, unlike the dual setting of simultaneous approximation (see \cite[Eq. (5.3)]{poels2022simultaneous}), it is possible that the sequence $(P_i)_{i\geq j}$ is contained in a proper subspace of $\bR[X]_{\leq n}$, see \eg \cite{moshchevitin2007best}. However, we will show later that under the hypothesis $d \asymp n^{1/3}$, we have $m_n\gg n^{1/3}$. The next definition is a dual version of \cite[Definition~5.2]{poels2022simultaneous}. Even though it is not important for our purpose, it is interesting to note that in \cite[Definition~5.2]{poels2022simultaneous}, the spaces $A_j[i]$ are generated by the points $\bx_i,\bx_{i+1},\dots$ coming \textsl{after} the good approximation $\bx_i$, whereas in the present setting we need to consider the points $Q_{i}, Q_{i-1},\dots$ coming \textsl{before} $Q_i$. It does not seem to work well the other way round.

\begin{Def}
    \label{Def: j_0,j_1,sigma_j(i) et Aj[i]}
    Let $j_1 > j_0\geq 0$ be such that
    \[
        \dim \Vect[\bR]{Q_{j_0},Q_{j_0+1},\dots, Q_{j_1}} = \dim \Vect[\bR]{Q_{j_0},Q_{j_0+1},\dots} = m_n.
    \]
    For each $i\geq j_1$ and $j=0,\dots, m_n -2$, we define
    \[
        \sigma_j(i) = k, \quad A_j[i] = \Vect[\bR]{Q_k, Q_{k+1},\dots, Q_i} \AND Y_j(i) = \normH{Q_{k-1}},
    \]
    where $k\in\{j_0+1,\dots,i\}$ is the smallest index such that $\dim \Vect[\bR]{Q_k,\dots, Q_i} = j+1$.
\end{Def}

Proposition~\ref{Cor: V_2n-k =  tout l'espace} implies that
    \begin{equation}
        \label{eq: V_2n-j(A_j[i]) = tout l'espace}
        V_{2n-j}\big(A_j[i]\big) = \bR[X]_{\leq 2n-j} \qquad (j=1,\dots,m_n-2).
    \end{equation}

\begin{Def}
    \label{Def: tau_j}
    Let $\tau\in(0,1)$. We associate to $\tau$ a sequence $(\tau_j)_{0\leq j\leq n/2}$ by setting $\tau_0 = \tau$, and for $j=1,\dots,\lfloor n/2 \rfloor$
    \begin{align*}
        \tau_{j} = \alpha_j\bigg(\tau_{j-1} - \frac{2j-1}{2n-d} \bigg), \quad \textrm{where } \alpha_j = \frac{(2n-d)\tau^2}{(n-2j)\tau+n-j+1}.
    \end{align*}
\end{Def}

The main result of this section is the following. The second part of the proposition, which we will use later to get a lower bound for $m_n$, will be proved thanks to Corollary~\ref{Cor: tau_m_2 < mu} below.

\begin{Prop}
    \label{Prop: prop resumee avec tau_j}
    Let $\tau\in(0,1)$ and let $(\tau_j)_{0\leq j \leq n/2}$ be as in Definition~\ref{Def: tau_j}. Suppose that
    \begin{equation}
    \label{eq: def exposant tau}
        \normH{Q_{i+1}}^\tau \leq \normH{Q_i} \qquad \textrm{for each sufficiently large $i$}.
    \end{equation}
    Then for each large enough $i$, we also have
    \begin{align}
        \label{eq Prop: minoration de Y_j}
        \normH{Q_i}^{\tau_j} \ll Y_j(i) \qquad \textrm{for } j=0,\dots, \min\Big\{\lfloor n/2\rfloor,m_n-2\Big\},
    \end{align}
    with implicit constants which do not depend on $i$ and $j$.
\end{Prop}

\noindent\textsl{Remark.} We will use the exponent $\tau$ given in \eqref{eq: taille de H(Q_i+1)}. We will prove that under suitable conditions, the exponent of best approximation $\omega_n(\xi)$ is not ``too large'', which ensures that $\tau$ is ``close'' to $1$. This question, which is one of the delicate parts of this paper, will be dealt with in Section~\ref{section: estimation de omega_n}.

\medskip

In order to get \eqref{eq Prop: minoration de Y_j}, we will try to adopt a strategy similar to the one of \cite[§5]{poels2022simultaneous} in the setting of simultaneous approximation to the successive powers of $\xi$. New difficulties arise however, for example we need to work with $\HHstar$ instead of the standard height of subspaces (see Section~\ref{section: hauteurs tordues}). Schmidt's inequality \eqref{eq: Schmidt's inequality HHstar} will play a key-role in our proofs. We keep the notation of Definition~\ref{Def: fonction g_A(N)} for the sets $\cB_k(\cA)$ and the subspaces $V_k(\cA) \subset \bR[X]_{\leq k}$.

\begin{proof}
    Without loss of generality, we may suppose that the index $j_0$ is large enough for us to have \eqref{eq: def exposant tau} for each $i\geq j_0-1$. Let us fix $i\geq j_1$, and for simplicity write $m=m_n$ and $Y_k:= Y_k(i)$ for $k=0,\dots,m-2$.

    \medskip

    We prove \eqref{eq Prop: minoration de Y_j} by induction on $j$. If $j=0$, then we have $Y_0 = \normH{Q_{i-1}}$ since $\sigma_0(i) = i$. By \eqref{eq: def exposant tau} applied with $i-1$ instead of $i$, we get $\normH{Q_i}^{\tau_0} \leq Y_0$. Now, let $j\in\{1,\dots,m-2\}$ with $j\leq n/2$ such that \eqref{eq Prop: minoration de Y_j} holds for $j-1$. If $\tau_{j}\leq 0$, then \eqref{eq Prop: minoration de Y_j} holds trivially for $j$. We assume that $\tau_{j} > 0$. In particular, we also have $\tau_{j-1} > 0$. Write $P:= Q_{\sigma_j(i)}$ and $Q:= Q_{\sigma_j(i)+1}$. By \eqref{eq: def exposant tau}, we have
    \begin{align}
        \label{eq proof: contrôle H(P) et H(Q) par Y_j}
        \normH{Q}^{\tau^2} \leq \normH{P}^{\tau} \leq Y_j.
    \end{align}
    Since $P$ and $Q$ are coprime, Lemma~\ref{lem : dim V_n+j(P,Q)} implies that $\dim V_{2n-j}(P,Q)\geq 2(n-j+1)$. Therefore, there exists a family of $2n-3j+1$ linearly independent polynomials
    \[
        \cU_j:=\{U_0,\dots,U_{2n-3j}\} \subset \cB_{2n-j}(P,Q)
    \]
    such that $A_j[i] \cap \Vect[\bR]{\cU_j} = \{0\}$. Note that since $j\leq n/2$, we may choose $\cU_j$ such that it contains at least $n-2j$ polynomials whose height is equal to $\normH{P}$. The remaining $n-j+1$ ones have height~$\leq \normH{Q}$. By \eqref{eq: V_2n-j(A_j[i]) = tout l'espace}, we have $V_{2n-j}\big(A_j[i]\big) = \bR[X]_{\leq 2n-j}$. Therefore, there exists
    \[
        \cV_j:=\{V_1,\dots,V_{j-1}\} \subset \cB_{2n-j}(Q_{\sigma_j(i)},\dots, Q_i)
    \]
    (with the convention $\cV_j = \emptyset$ if $j=1$) such that we have the direct sum
    \[
        A_j[i] \oplus \Vect[\bR]{\cU_j} \oplus \Vect[\bR]{\cV_j} = \bR[X]_{\leq 2n-j}.
    \]
    All the polynomials of $\cV_j$ have height at most $\normH{Q_i}$. Let $k\in\{\sigma_j(i),\dots,i\}$ which maximizes $|Q_k(\xi)|/\normH{Q_k}$ and define
    \[
        A:= A_j[i] \AND B:= \Vect[\bR]{\cU_j \cup \cV_j \cup \{Q_k\}},
    \]
    so that $A + B = \bR[X]_{\leq 2n-j}$ and $A\cap B = \Vect[\bR]{Q_k}$. We will now make a crucial use of the function $\HHstar$ introduced in Definition~\ref{Def: def HHstar} (here, the ambient space is  $\bR[X]_{\leq 2n-j}$, identified to $\bR^{2n-j+1}$ via \eqref{eq: identification poly avec points de R^m}). Recall that
     \[
        \HHstar(A+B) = \HHstar\big(\bR[X]_{\leq 2n-j}\big) = \norm{(1,\xi,\dots,\xi^{2n-j})} \asymp 1,
     \]
     and that according to \eqref{eq: HHstar(P)} the primitive polynomial $Q_k$ satisfies
     \[
        \HHstar(A\cap B) = \HHstar\big(\Vect[\bR]{Q_k}\big) =  |Q_k(\xi)|.
     \]
     Schmidt's inequality \eqref{eq: Schmidt's inequality HHstar} applied with the subspaces $A$ and $B$ yields
    \begin{align}
        \label{eq proof: Schmidt notre construction}
        |Q_{k}(\xi)| \asymp \HHstar(A+B)\HHstar(A\cap B) \leq \HHstar(A) \HHstar(B),
    \end{align}
    the implicit constants depending only on $n$ and $\xi$ (and not on the indices $i,j$). It remains to estimate $\HHstar(A)$ and $\HHstar(B)$. The subspace $B\subset\bR[X]_{\leq 2n-j}$ is generated by the $2n-2j+1$ linearly independent polynomials $\cV = \cU_j \cup \cV_j \cup \{Q_k\}$. Moreover (see the remarks after the constructions of $\cU_j$ and $\cV_j$), we have
    \[
        \prod_{R\in\cV}\normH{R} \leq \normH{P}^{n-2j}\normH{Q}^{n-j+1}\normH{Q_i}^{j-1}\normH{Q_k}.
    \]
    By choice of $k$, for each $R\in\cV$ we also have $|R(\xi)|/\normH{R}\ll |Q_k(\xi)|/\normH{Q_k}$, and Lemma~\ref{lem: estimation des pseudo resultant tordus} combined with the above yields the upper bound
    \begin{align*}
        \HHstar(B) \ll |Q_k(\xi)| \normH{P}^{n-2j}\normH{Q}^{n-j+1}\normH{Q_i}^{j-1}.
    \end{align*}
    The space $A_j[i]\subset\bR[X]_{\leq 2n-j}$ is spanned by a set $\cU$ of $j+1$ linearly polynomials that may be chosen among $Q_{\sigma_{j-1}(i)-1}$,...,$Q_{i-1}$, $Q_i$. For each $R\in\cU$, we have $\normH{R}\leq \normH{Q_i}$ and $|R(\xi)| \leq \normH{R}^{-\homega} \leq Y_{j-1}^{-\homega}$. Combined with Lemma~\ref{lem: estimation des pseudo resultant tordus}, we obtain
    \begin{align*}
        \HHstar(A)  \ll \sum_{R\in\cU}|R(\xi)|\prod_{\substack{S\in\cU \\ S\neq R}}\normH{S} \ll Y_{j-1}^{-\homega}\normH{Q_{i}}^j.
    \end{align*}
    Then, combining the above upper bounds for $\HHstar(B)$ and $\HHstar(A)$ with \eqref{eq proof: Schmidt notre construction} and \eqref{eq proof: contrôle H(P) et H(Q) par Y_j}, we get
    \begin{align*}
        Y_{j-1}^{\homega}  \ll \normH{P}^{n-2j}\normH{Q}^{n-j+1}\normH{Q_i}^{2j-1} & \ll Y_j^{(n-2j)/\tau + (n-j+1)/\tau^2}\normH{Q_i}^{2j-1},
    \end{align*}
    where the implicit constants depend on $n$ and $\xi$ only. Using the induction hypothesis, we also have $\normH{Q_i}^{\homega\,\tau_{j-1}} \ll Y_{j-1}^{\homega}$, hence
    \begin{align*}
        \normH{Q_i}^{\homega\,\tau_{j-1}-2j+1} \ll  Y_j^{(n-2j)/\tau + (n-j+1)/\tau^2} = Y_j^{(2n-d)/\alpha_j}.
    \end{align*}
    Rising each term to the power $\alpha_j/(2n-d)$ and using $\homega > 2n-d$ we easily deduce \eqref{eq Prop: minoration de Y_j} for $j$, which concludes our induction step.
\end{proof}

\begin{Rem}
    We could get an exponent $\tau_j$ a little bit greater in the above proposition (by giving a slightly better estimate of $\HHstar(A)$ is the proof). However, those improvements would only lead to a larger constant $a$ in Theorem~\ref{Thm : main} at best; the term $n^{1/3}$ would remain the same, whereas we are expecting $n^{1/2}$. We preferred to keep the arguments simple.
\end{Rem}

\begin{Prop}
    \label{Prop: majoration Y_m-2}
    Let the hypotheses be as in Proposition~\ref{Prop: prop resumee avec tau_j} and write  $m=m_n$. For any $\lambda < \lambda_n(\xi)$, there are infinitely many indices $i$ such that
    \begin{equation*}
        \label{eq: majoration theta_m-2 eq 0}
        Y_{m-2}(i) \leq \normH{Q_i}^{1/(\homega\lambda\tau)}.
    \end{equation*}
    In particular, there are infinitely many indices $i$ such that
    \begin{equation}
        \label{eq: majoration theta_m-2}
        Y_{m-2}(i) \leq \normH{Q_i}^\mu,\quad \textrm{where } \mu := \frac{n}{(2n-d)\tau}.
    \end{equation}
\end{Prop}

\begin{proof}
    By definition of $m$, the subspace
    \begin{equation}
        \label{eq proof: def V}
        V = \Vect[\bR]{Q_{\sigma_{m-2}(i)-1},Q_{\sigma_{m-2}(i)},\dots,Q_{i}}
    \end{equation}
    of $\bR[X]_{\leq n}$, does not depend on $i$, for each $i\geq j_1$, where $j_1$ comes from Definition~\ref{Def: j_0,j_1,sigma_j(i) et Aj[i]}. It has dimension $m$ since $\dim A_{m-2}[i] = m-1$ and $Q_{\sigma_{m-2}(i)-1}\notin A_{m-2}[i]$. Fix two positive real numbers $\alpha,\lambda$ with $\lambda < \alpha < \lambda_n(\xi)$, and suppose by contradiction that there exists an index $i_0\geq j_1$ such that for each $i \geq i_0$
    \begin{equation}
        \label{eq proof: def theta}
        Y_{m-2}(i) \geq \normH{Q_i}^\theta, \quad \textrm{where } \theta=\frac{1}{\homega\lambda\tau}.
    \end{equation}
    By hypothesis, we can also assume that $\normH{Q_{i+1}}^\tau \leq \normH{Q_i}$ for each $i\geq i_0$. Identifying $\bR[X]_{\leq n}$ to $\bR^{n+1}$ via the isomorphism \eqref{eq: identification poly avec points de R^m}, we claim that the point $\Xi=(1,\xi,\xi^2,\dots,\xi^n)$ is orthogonal to $V$, with respect to the standard scalar product $\psc{\cdot}{\cdot}$ of $\bR^{n+1}$.

    \medskip

     By definition of $\lambda_n(\xi)$ there exist infinitely many non-zero $\by=(y_0,\dots,y_n)\in\bZ^{n+1}$ satisfying
    \begin{equation*}
         L(\by) =\max_{1\le k \le n} |y_0\xi^k - y_k| \leq Y^{-\alpha},\quad \textrm{where } Y = \normH{\by} = \max_{1\leq k\leq n} |y_k|.
    \end{equation*}
    Let $(\by_i)_{i\geq 0}$ be an unbounded sequence of such points ordered by increasing norm.  This sequence converges projectively to $\Xi = (1,\xi,\xi^2,\dots,\xi^n)$. Without loss of generality, we may assume that $(\normH{\by_0})^{\alpha} > 2(n+1)\normH{Q_{i_0}}$. Fix an index $j$ arbitrarily large. For simplicity, set $\by:= \by_j$ and $Y = \normH{\by_j}$. There exists an index $i\geq i_0$ such that
    \begin{align}
    \label{eq proof: sur l'exposant mu eq1}
         \normH{Q_i} < \frac{Y^{\alpha}}{2(n+1)} \leq \normH{Q_{i+1}} \leq \normH{Q_i}^{1/\tau}.
    \end{align}
    Note that $i$ tends to infinity as $j$ tends to infinity.  Let $k\in\{\sigma_{m-2}(i)-1,\dots,i\}$. The polynomial $Q:=Q_k$ is identified with an integer point $\bz\in\bZ^{n+1}$ such that $Q(\xi)= \psc{\bz}{\Xi}$. Since $\psc{\bz}{\by} = \psc{\bz}{\by -y_0\Xi} + y_0\psc{\bz}{\Xi}$, we get
    \begin{align*}
        |\psc{\bz}{\by}| \leq (n+1)\normH{Q}L(\by) + Y |Q(\xi)|
    \end{align*}
    (this argument is similar to the one used by Laurent in the proof of \cite[Lemma~5]{laurent2003simultaneous}). Our hypothesis \eqref{eq proof: def theta} yields
    \begin{align*}
        \normH{Q_i}^\theta \leq Y_{m-2}(i)  \leq \normH{Q} \leq \normH{Q_i}.
    \end{align*}
    Using \eqref{eq proof: sur l'exposant mu eq1} together with $L(\by)\leq Y^{-\alpha}$, we get
    \begin{align*}
        (n+1)\normH{Q}L(\by) <\frac{1}{2}.
    \end{align*}
    Moreover, \eqref{eq proof: sur l'exposant mu eq1} also yields $ Y^{1/\homega} \ll \normH{Q_i}^{1/(\homega\alpha\tau)} = \normH{Q_i}^{\theta \lambda / \alpha}$, where the implicit constant only depends on $n$. Since $\lambda < \alpha$, we may choose $j$ so large that $(2Y)^{1/\homega} < \normH{Q_i}^\theta$. Combining this with the estimate $|Q(\xi)|\leq \normH{Q}^{-\homega}$ from \eqref{eq: petitesse de Q_i(xi)}, we also get
    \[
        Y|Q(\xi)|\leq Y\normH{Q}^{-\homega} \leq Y\normH{Q_i}^{-\theta\,\homega} < \frac{1}{2}.
    \]
    We conclude that the integer $|\psc{\bz}{\by}|$ is (strictly) less that $1$. It is thus equal to $0$, and so $\by$ and $\bz$ are orthogonal. By letting $\bz$ vary, this implies that $\by = \by_j$ is orthogonal to the subspace $V$. Since this is true for all sufficiently large $j$, it follows that the (projective) limit $\Xi$ is also orthogonal to $V$. This proves our claim and provides the required contradiction since no $Q_i$ vanishes at the transcendental number $\xi$.  Thus, \eqref{eq proof: def theta} does not hold for arbitrarily large indices $i$. Estimate \eqref{eq: majoration theta_m-2} follows by noticing that $\lambda_n(\xi)\geq 1/n$ by Dirichlet's theorem, and $\homega > 2n-d$. We may therefore choose $\lambda < \lambda_n(\xi)$ so that $\lambda\homega > (2n-d)/n$.
\end{proof}

\begin{Cor}
    \label{Cor: tau_m_2 < mu}
    Under the same hypotheses, suppose moreover that $m=m_n$ satisfies $m-2\leq n/2$, and let $(\tau_j)_{0\leq j \leq n/2}$ be as in Definition~\ref{Def: tau_j}. Then, we have
    \[
        \tau_{m-2} \leq \mu = \frac{n}{(2n-d)\tau}.
    \]
\end{Cor}

\begin{proof}
    By Proposition~\ref{Prop: majoration Y_m-2} combined with Proposition~\ref{Prop: prop resumee avec tau_j} there are infinitely many indices $i$ for which $\normH{Q_i}^{\tau_{m-2}} \ll Y_{m-2}(i) \leq \normH{Q_i}^\mu$. Since $\normH{Q_i}$ tends to with $i$, we deduce that $\tau_{m-2}\leq \mu$.
\end{proof}

\section{Upper bound on the exponent of best approximation}
\label{section: estimation de omega_n}

Our goal is the following result, which we will prove at the end of this section.

\begin{Prop}
    \label{Prop: estimation omega_n(xi)}
    Suppose that $\homega_n(\xi) > 2n-d$, with an integer $d\in\bN$ satisfying $\displaystyle 2 \leq d\leq \sqrt[3]{n/4}$. Then, we have the upper bound
    \begin{align*}
        \omega_n(\xi)\leq 2n + P(n,d), \quad\textrm{where } P(n,d) = \frac{n(4d^2-d-5)+8d^2-2d-15}{2n-8d^2+2d+15}.
    \end{align*}
    If moreover we have $d\leq \Big\lceil\sqrt[3]{n/16} \Big\rceil$ and $n > 16$, then
    \begin{align*}
        \omega_n(\xi)\leq 2n + 2d^2.
    \end{align*}
\end{Prop}

Let $d,n,\xi$ and $\homega$ be as in Sections~\ref{section: sequence des Q_i} and~\ref{section : indpce lineaire Q_i}. We suppose thus that $2\leq d < 1+n/2$ and that \eqref{eq: notation homega} holds, namely
\[
    \homega_n(\xi) > \homega > 2n -d.
\]
Fix a sequence of minimal polynomials $(P_i)_{i\geq 0}$ associated to $n$ and $\xi$ as in Section~\ref{section: minimal pol}. We denote by $(Q_i)_{i\geq 0}$ the sequence of irreducible factors given by Proposition~\ref{prop: existence Q_i}. Unless otherwise stated, all the constants implicit in the symbols  $\ll$, $\gg, \asymp$ and $\GrO(\cdot)$ only depend on $n$, $d$, $\xi$ and $\homega$.

\begin{Rem}
    According to Proposition~\ref{prop: existence Q_i}, we have $\omega_n(\xi) = \limsup_{i\rightarrow\infty}\omega(Q_i)$. Also by \eqref{eq: premiere estimation Pj vs Qi}, we have
\begin{equation}
    \label{eq: rappel theta_i}
    \normH{P_j} \leq \normH{Q_i}^{1+\theta_i}, \quad \textrm{where } \theta_i = \frac{\omega(Q_i)-2n+d}{n-2d+3},
\end{equation}
for each $i\geq 0$ and each $j$ such that $Q_i$ divides $P_j$. Proposition~\ref{Prop: estimation omega_n(xi)} implies that if $d^3$ is small compared to $n$, then $\theta_i = \GrO(d^2/n)$ is also small, and $Q_i$ has ``almost'' the same height as $P_j$.
\end{Rem}
In order to bound from above $\omega_n(\xi)$, it suffices to do so for $\omega(Q_i)$. We could try to use \eqref{eq: rappel theta_i}, which implies that any minimal polynomial of height greater than $\normH{Q_i}^{1+\theta_i}$ is not divisible by $Q_i$. They are thus coprime and we may consider their (non-zero) resultant. However we cannot conclude, as $\theta_i$ is too large. To solve this problem, we need several lemmas. We first start by a few simple observations. A quick computation yields
\begin{equation}
    \label{equation fonctionnelle theta eq 1}
    (1+\theta_i)(2n-d) = \omega(Q_i) + (n+d-3)\theta_i.
\end{equation}
More generally, for each $\eta\geq 0$, 
we have
\begin{equation}
    \label{equation fonctionnelle theta eq 2}
    \big[1+\theta_i(1-\eta) \big](2n-d) = \omega(Q_i) + \big( n+d-3 - \eta(2n-d) \big)\theta_i.
\end{equation}
Under the condition $\eta < (n+d-3)/(2n-d)$, which holds as soon as $\eta < 1/2$, it implies that for each $i\geq 0$, we have
\begin{equation}
    \label{equation: minoration de Q_i(xi) via theta_i}
    |Q_i(\xi)| = \normH{Q_i}^{-\omega(Q_i)} >  \normH{Q_i}^{-\big(1+\theta_i(1-\eta)\big)(2n-d)}.
\end{equation}

\begin{Lem}
    \label{Lem: P = QR grand bonne approx}
    Let $i\geq 0$ and $\eta\in[0,1/2)$, and suppose that $R\in\bZ[X]_{\leq d-2}$ is a non-zero polynomial such that $P := Q_iR$ has degree at most $n$, and $P$ is solution of
    \begin{equation}
        \label{eq : systeme dio pour P = QR}
        \normH{P} \leq H:= \normH{Q_i}^{1+\theta_i(1-\eta)} \AND |P(\xi)| \leq H^{-2n+d}.
    \end{equation}
    Define
    \begin{align*}
        \eta' = \frac{(2n-d)\eta}{n+d-3} \AND \eta'' = \frac{(2n-2d+3)\eta+d-3}{n+d-3}.
    \end{align*}
    Then, we have the following properties.
    \begin{enumerate}[label=(\roman*)]
      \item \label{item: item 1 Lem P = QR sol }
        The polynomial $R$ is non-constant. We have $d\geq 3$ and
            \begin{align}
                \label{eq Prop: P = QR grand bonne approx: R(xi) revisite}
                \normH{R}^{-(n+d-3)} \ll |R(\xi)| \leq \normH{Q}^{-(n+d-3)(1-\eta')\theta_i}.
            \end{align}
      \item \label{item: item 2 Lem P = QR sol }
            There exist a non-constant irreducible polynomial $A\in\bZ[X]_{\leq n}$ and an integer $e \in[1,d-2]$ such that $A^e$ divides $R$,
            \begin{align}
                \label{eq Prop: P = QR grand bonne approx: H(A^e) revisite}
                \normH{A^e} \gg \normH{Q}^{\theta(1-\eta'')} \AND \normH{A^e}^{-(n+d-3)} \ll |A^e(\xi)|.
            \end{align}
        \item If $S\in\bZ[X]_{\leq d-2}$ is non-zero polynomial such that $A$ and $S$ are coprime and $\normH{S}\leq \normH{A^e}$, then
            \begin{align}
                \label{eq Prop: P = QR grand bonne approx: H(S) revisite}
                |S(\xi)| \gg \normH{A^e}^{-(2d-5)}.
            \end{align}
    \end{enumerate}
\end{Lem}

\begin{proof}
    Fix $i\geq 0$. For simplicity, write $Q:= Q_i$ and $\theta = \theta_i$. By Gelfond's Lemma, we have
    \begin{align*}
        \normH{Q}\normH{R}\asymp \normH{QR} = \normH{P} \leq \normH{Q}^{1+\theta(1-\eta)},
    \end{align*}
    so that
    \begin{equation}
        \label{eq proof: H(R) << fonction de H(Q)}
        \normH{R}\ll \normH{Q}^{\theta(1-\eta)}.
    \end{equation}
    The first inequality of \eqref{eq Prop: P = QR grand bonne approx: R(xi) revisite} and the seconde one of \eqref{eq Prop: P = QR grand bonne approx: H(A^e) revisite} are consequences of Lemma~\ref{lemme: omega(R) petit degre} (using $\deg (A^e) \leq \deg(R) \leq d-2$). Using \eqref{equation fonctionnelle theta eq 2} together with \eqref{eq : systeme dio pour P = QR} and $\normH{Q}^{-\omega(Q)} = |Q(\xi)|$, we find
    \begin{align*}
        |Q(\xi)R(\xi)| = |P(\xi)|  &\leq \normH{Q}^{-\big[1+\theta(1-\eta)\big](2n-d)} =  |Q(\xi)|\normH{Q}^{- \big( n+d-3 - \eta(2n-d) \big)\theta}.
    \end{align*}
    Simplifying by $|Q(\xi)|$ yields the second inequality of \eqref{eq Prop: P = QR grand bonne approx: R(xi) revisite}. In particular we have $|R(\xi)| < 1$ since $\normH{Q} > 1$ (and $\theta > 0$ as well as $\eta'\leq 2\eta < 1$). Consequently $R\in\bZ[X]_{\leq d-2}$ cannot be constant, and thus $d\geq 3$.

    \medskip

    Without loss of generality, we may suppose that $P$ (and thus $R$) is primitive. Let us consider the factorization of $R$ over $\bZ$. There exist an integer $k\geq 1$ and irreducible (non-constant) pairwise distinct polynomials $A_1,\dots,A_k\in\bZ[X]$ such that
    \[
        R = \prod_{j=1}^{k} A_j^{\alpha_j} = \prod_{j=1}^{k} B_j \qquad \textrm{with $B_j:= A_j^{\alpha_j}$ for each $j=1,\dots,k$},
    \]
    and where $\alpha_1,\dots,\alpha_k$ are positive integers. According to Lemma~\ref{lem: lem general 2}, there exists $j\in\{1,\dots,k\}$ such that $B = B_j$ satisfies
    \begin{align*}
        |B(\xi)| \ll \normH{R}^{d-3}|R(\xi)|.
    \end{align*}
    We use \eqref{eq proof: H(R) << fonction de H(Q)} to bound  $\normH{R}$ from above, and the second inequality of~\eqref{eq Prop: P = QR grand bonne approx: R(xi) revisite} to bound $|R(\xi)|$ from above. Then, Lemma~\ref{lemme: omega(R) petit degre} applied to the polynomial $B\in\bZ[X]_{\leq d-2}$ together with the above yields
    \begin{align*}
        \normH{B}^{-(n+d-3)} \ll |B(\xi)| \ll  \normH{R}^{d-3}|R(\xi)| & \ll \normH{Q}^{(d-3)(1-\eta)\theta-(n+d-3)(1-\eta')\theta},
    \end{align*}
    Since by definition of $\eta'$ and $\eta''$ we have
    \begin{align*}
        1-\eta'-\frac{d-3}{n+d-3}(1-\eta) = 1-\eta'',
    \end{align*}
    we deduce that
    \begin{equation}
        \label{eq proof: B(xi) << fonction de H(Q)}
        \normH{B}^{-(n+d-3)} \ll |B(\xi)| \ll \normH{Q}^{-(n+d-3)\theta(1-\eta'')}.
    \end{equation}
    and \eqref{eq Prop: P = QR grand bonne approx: H(A^e) revisite} follows easily upon recalling that $A^e = B$. Now, suppose that $S\in\bZ[X]_{\leq d-2}$ is a non-zero polynomial coprime to $A$ with $\normH{S}\leq \normH{B}$. If $S$ is constant, then \eqref{eq Prop: P = QR grand bonne approx: H(S) revisite} is trivial. We may therefore assume that $S$ has degree at least $1$. Then, the estimate of Lemma~\ref{lem: estimation classique du resultant} yields
    \begin{align}
        1 \leq |\Res(B,S)| & \ll \normH{B}^{d-3}\normH{S}^{d-2}|B(\xi)| + \normH{B}^{d-2}\normH{S}^{d-3}|S(\xi)| \notag \\
        & \ll  \normH{B}^{2d-5}\big(|B(\xi)|+|S(\xi)| \big) \label{eq proof: estimation resultant}
    \end{align}
    (where the implicit constants depend on $\xi$, $n$ and $c$). Recall that $B$ divides $R$, we therefore have $\normH{B}\ll \normH{R}$. Together with \eqref{eq proof: H(R) << fonction de H(Q)}, it gives $\normH{B} \ll \normH{Q}^{\theta(1-\eta)}$. Combining the above with \eqref{eq proof: B(xi) << fonction de H(Q)}, we obtain
    \begin{align*}
        \normH{B}^{2d-5}|B(\xi)| \ll \normH{Q}^{(2d-5)\theta(1-\eta)-(n+d-3)\theta(1-\eta'')}.
    \end{align*}
    On the other hand, using $\eta\leq 1/2$ we get
    \begin{align*}
        (2d-5)(1-\eta)-(n+d-3)(1-\eta'') &= (2n-4d+8)\eta - (n-2d+5) \leq -1.
    \end{align*}
    Since for each large enough $i$, the number $\theta = \theta_i$ is bounded from below by
    \[
        \rho = \frac{\homega-2n+d}{n-2d+3} > 0,
    \]
    it follows that $ \normH{B}^{2d-5}|B(\xi)| \ll \normH{Q}^{-\rho}$ tends to $0$ as $i$ tends to infinity. Consequently, \eqref{eq proof: estimation resultant} becomes
    \begin{align*}
        1 \ll  \normH{B}^{2d-5}|S(\xi)|,
    \end{align*}
    hence \eqref{eq Prop: P = QR grand bonne approx: H(S) revisite}.
\end{proof}

\begin{Lem}
    \label{Prop-clef: existence d'un bon point min successeur}
    Let $\eta\in[0,1/2)$. As in Lemma~\ref{Lem: P = QR grand bonne approx}, we set
    \[
        \eta'' = \frac{(2n-2d+3)\eta+d-3}{n+d-3}.
    \]
    Suppose that either $d=2$, or $d\geq 3$ and we have $\eta''\in [0,1/2)$ as well as
    \begin{align}
        \label{eq Prop: inegalite pour eta}
        \frac{1-2\eta''}{1-\eta''} \geq 1-\frac{1}{d-2}  +  \frac{2d}{n}.
    \end{align}
    Then for each large enough $i\geq 0$, there exist $Z\in \bR$ with $\normH{Q_i}\leq Z \leq \normH{Q_i}^{1+\theta_i(1-\eta)}$ and a non-zero $P\in\bZ[X]_{\leq n}$, coprime to $Q_i$, which satisfies  $|P(\xi)| < |Q_i(\xi)|$ and  is solution of
    \begin{equation}
        \label{eq Prop: system sol premiere avec Q_i}
        \normH{P} \leq Z \AND |P(\xi)| \leq Z^{-(2n-d)}.
    \end{equation}
\end{Lem}

\begin{proof}
    Since $\homega_n(\xi) > 2n-d$, there exists $X_0\geq 0$ such that for each $X \geq X_0$ the system
    \begin{equation*}
        \normH{P} \leq X \AND |P(\xi)| \leq X^{-(2n-d)}
    \end{equation*}
    has a non-zero solution $P$ in $\bZ[X]_{\leq n}$. Fix $i\geq 0$ such that $\normH{Q_i}\geq X_0$, and a non-zero solution $P\in\bZ[X]_{\leq n}$ of the above system with $X := \normH{Q_i}^{1+\theta_i(1-\eta)}$. For simplicity, write $Q= Q_i$ and $\theta = \theta_i$. We have $|P(\xi)|\leq X^{-(2n-d)} < |Q(\xi)|$ thanks to \eqref{equation: minoration de Q_i(xi) via theta_i}. If $P$ and $Q$ are coprime, then the conclusion holds with $Z=X$. We may therefore assume that $P$ and $Q$ are not coprime. Then $Q$ divides $P$, and assertion~\ref{item: item 1 Lem P = QR sol } of Lemma~\ref{eq : systeme dio pour P = QR} implies that $d\geq 3$. Let $A\in\bZ[X]_{\leq d-2}$ and $e\in[1,d-2]$ be the non-constant irreducible polynomial and the integer given by Lemma~\ref{Lem: P = QR grand bonne approx}~\ref{item: item 2 Lem P = QR sol }. In particular we have $\deg(A^e)\leq d-2$ and \eqref{eq Prop: P = QR grand bonne approx: H(A^e) revisite} holds. Set $Z:= e^{-n}2\normH{QA^{e}}$, and define $\nu$  by the relation
    \[
        Z = \normH{Q}^{1+\theta(1-\nu)}.
    \]
    By Gelfond's Lemma and by definition of $Z$ and $\nu$, we have
    \begin{align*}
        \normH{Q}^{\theta(1-\nu)} \asymp \normH{A^e}  \gg \normH{Q}^{\theta(1-\eta'')},
    \end{align*}
    the last inequality coming from \eqref{eq Prop: P = QR grand bonne approx: H(A^e) revisite}. We deduce that $\nu \leq \eta'' + \GrO\big(1/\log \normH{Q}\big)$. Since $\eta'' < 1/2$ we may assume $i$ large enough so that $\nu < 1/2$. On the other hand, since $QA^e$ divides $P$, by \eqref{eq: Gelfond's Lemma}, we have
    \[
        Z < e^{-n}\normH{QA^e} \leq \normH{P} \leq X = \normH{Q}^{1+\theta(1-\eta)},
    \]
    hence $\nu\geq \eta$. We now consider a non-zero solution $\tP\in\bZ[X]_{\leq n}$ of the system
    \begin{equation}
    \label{eq proof: system avec Y}
        \normH{\tP} \leq Z \AND |\tP(\xi)| \leq Z^{-(2n-d)}.
    \end{equation}
    We claim that $\tP$ and $Q$ are coprime. Suppose by contradiction that $Q$ divides $\tP$. There exists $\tR\in\bZ[X]$ such that $\tP = Q\tR$. Write $\tR = A^f\tS$, with $f\in\bN$ and $\tS\in\bZ[X]_{\leq d-2}$ coprime to $A$. By \eqref{eq: Gelfond's Lemma} and by definition of $Z$, and since $Q$ and $\tS$ divide $\tP$, we obtain
    \[
        \normH{Q}\normH{\tS} < e^{n}\normH{\tP} \leq e^nZ = e^{-n} \normH{QA^e} < \normH{Q}\normH{A^e}.
    \]
    We deduce that $\normH{\tS}\leq \normH{A^e}$. Similarly,
    \[
        \normH{QA^f} < e^n\normH{\tP} \leq e^nZ = e^{-n} \normH{QA^e}.
    \]
    Consequently, the polynomial $QA^e$ cannot be a factor of $QA^f$ (by \eqref{eq: Gelfond's Lemma} once again). Thus $f\leq e-1$. Since $\normH{\tS}\leq \normH{A^e}$, the last assertion of Lemma~\ref{Lem: P = QR grand bonne approx} yields
    \begin{align}
        \label{eq proof : estimation de S(xi)}
        |\tS(\xi)| \gg \normH{A^e}^{-(2d-5)}.
    \end{align}
    By hypothesis $\nu < 1/2$, and Lemma~\ref{Lem: P = QR grand bonne approx}~\ref{item: item 1 Lem P = QR sol } applied to the solution $\tP=Q\tR$ of the system \eqref{eq proof: system avec Y} gives the estimate
     \begin{equation}
        \label{eq proof: minoration omega(S)log H(S)}
         |\tR(\xi)| \leq \normH{Q}^{-(n+d-3)(1-\nu')\theta}, \quad \textrm{where } \nu'= \frac{(2n-d)\nu}{n+d-3}.
    \end{equation}
    We now use \eqref{eq proof : estimation de S(xi)} and $|A^e(\xi)| \gg \normH{A^e}^{-(n+d-3)}$ (coming from \eqref{eq Prop: P = QR grand bonne approx: H(A^e) revisite}) together with  $f\leq e-1\leq d-3$. We get the lower bound
    \begin{align*}
        \log|\tR(\xi)| & = \frac{f}{e}\log|A^e(\xi)|+\log|\tS(\xi)| \\
        & \geq -\Big[\Big(1-\frac{1}{e} \Big)(n+d-3) +  2d-5\Big]\log \normH{A^e} + \GrO(1)\\
        & \geq  -\Big[\Big(1-\frac{1}{d-2} \Big)(n+d-3) +  2d-5\Big]\theta(1-\nu)\log \normH{Q} + \GrO(1),
    \end{align*}
    the last inequality following from $\normH{A^e} \asymp \normH{Q}^{\theta(1-\nu)}$. Comparing this with \eqref{eq proof: minoration omega(S)log H(S)} and noting that $\nu'\leq 2\nu$, we obtain
    \begin{align*}
        \frac{1-2\nu}{1-\nu} \leq \frac{1-\nu'}{1-\nu} \leq 1-\frac{1}{d-2}  +  \frac{2d-5}{n+d-3} + \GrO\big(1/\log \normH{Q}\big).
    \end{align*}
    The function $ \nu \mapsto (1-2\nu)/(1-\nu)$ is decreasing on $[0,1/2]$. Using the estimate $\nu \leq \eta'' + \GrO\big(1/\log \normH{Q}\big)$, we obtain
    \begin{align*}
        \frac{1-2\eta''}{1-\eta''} \leq 1-\frac{1}{d-2}  +  \frac{2d-5}{n+d-3} + \GrO\big(1/\log \normH{Q}\big).
    \end{align*}
    Since $(2d-5)/(n+d-3) < 2d/n$, this contradicts our hypothesis \eqref{eq Prop: inegalite pour eta} when $i$ is sufficiently large. So, if $i$ is large enough, then $\tP$ and $Q$ are coprime. Finally, the lower bound $|\tP(\xi)|\leq Z^{-(2n-d)} < |Q(\xi)|$ follows from \eqref{equation: minoration de Q_i(xi) via theta_i} with $\eta$ replaced by $\nu$ (since $\nu < 1/2$), by a similar argument as in the beginning of the proof.
\end{proof}


\begin{proof}[Proof of Proposition~\ref{Prop: estimation omega_n(xi)}]
    The condition $d\leq \sqrt[3]{n/4}$ implies that $d\leq 1+n/2$. Set $\eta = 1/(2d+5/2)$. Note that the upper bound $2n+P(n,d)$ is not optimal in Proposition~\ref{Prop: estimation omega_n(xi)} (and could be slightly improved by choosing the parameter $\eta$ closer to $1/(2d)$). Suppose that $d\geq 3$ and define
    \begin{align*}
        \eta'' =  \frac{(2n-2d+3)\eta+d-3}{n+d-3} \AND \nu = \frac{1}{d+1}.
    \end{align*}
    A direct computation yields
    \begin{align*}
        \eta'' - \nu = \frac{-n+4d^3-11d^2-13d+6}{(4d+5)(n+d-3)(d+1)} < 0,
    \end{align*}
    so that $\eta'' < \nu \leq 1/3$. Since $x\mapsto (1-2x)/(1-x)$ is decreasing on $[0,1/2]$, we deduce that $(1-2\eta'')/(1-\eta'') \geq (1-2\nu)/(1-\nu)$. On the other hand, if $d\geq 3$ we have
    \begin{align*}
        \frac{1-2\nu}{1-\nu} - \Big( 1-\frac{1}{d-2}  +  \frac{2d}{n}\Big) = \frac{2(n-d^3+2d^2)}{nd(d-2)} \geq 0.
    \end{align*}
    The conditions of Lemma~\ref{Prop-clef: existence d'un bon point min successeur} are therefore fulfilled for $d\geq 2$ (they are automatic for d=2). Consequently, for each large enough $i$ there exists a non-zero polynomial $P\in\bZ[X]_{\leq n}$ coprime with $Q_i$, satisfying
    \[
        |P(\xi)| \leq |Q_i(\xi)| < 1 \AND \normH{P}\leq \normH{Q_i}^{1+\theta(1-\eta)}.
    \]
    Such a polynomial is non-constant, and Lemma~\ref{lem: estimation classique du resultant} yields
    \begin{align*}
        1 \leq |\Res(Q_i,P)| &\ll \normH{Q_i}^{n-1}\normH{P}^{n}|Q_i(\xi)| + \normH{Q_i}^{n}\normH{P}^{n-1}|P(\xi)| \\
        & \ll \normH{Q_i}^{n-1+n(1+\theta(1-\eta))-\omega(Q_i)}.
    \end{align*}
    As $\normH{Q_i}$ tends to infinity, it follows that
    \begin{align*}
        n-1+n(1+\theta(1-\eta))-\omega(Q_i) \geq \GrO\Big(1/\log \normH{Q_i} \Big).
    \end{align*}
    Using the definition \eqref{eq: rappel theta_i} of $\theta_i$, a direct computation leads us to the estimate
    \begin{align*}
        (n\eta-2d+3)\omega(Q_i) \leq 2\eta n^2-(3d+\eta d-5)n+2d-3 + \GrO\Big(1/\log \normH{Q_i} \Big).
    \end{align*}
    The hypothesis $d\leq \sqrt[3]{n/4}$ implies $n\eta-2d+3 > 0$. Thus, after simplification
    \begin{align*}
        \omega(Q_i) + \GrO\Big(1/\log \normH{Q_i} \Big)  &\leq \frac{2\eta n^2-(3d+\eta d-5)n+2d-3}{n\eta-2d+3} \\
                    & \; = 2n + \frac{n(d-1-\eta d)+2d-3}{n\eta-2d+3} = 2n + P(n,d),
    \end{align*}
    where $P(n,d)$ is defined as in the statement of Proposition~\ref{Prop: estimation omega_n(xi)} (and $\eta=1/(2d+5/2)$). We conclude that
    \[
        \omega_n(\xi) = \limsup_{i\rightarrow\infty} \omega(Q_i) \leq 2n+P(n,d).
    \]
    Set $Q(n,d) = (2n-8d^2+2d+15)(P(n,d)-2d^2)$. A direct computation yields
    \begin{align*}
        Q(n,d) = -n(d+5)+16d^4-4d^3-22d^2-2d-15.
    \end{align*}
    If $\displaystyle d\leq \sqrt[3]{n/16}$ we have $16d^4 \leq nd$, and therefore $Q(n,d)\leq 0$. We obtain $P(n,d)\leq 2d^2$, and consequently $\homega_n(\xi)\leq 2n+2d^2$. It remains to show that in the case $n\geq 17$ and $d= \lceil \rho \rceil$, where
    \[
        \rho = \sqrt[3]{n/16},
    \]
    we still have $Q(n,d)\leq 0$. If $17 \leq n \leq 128$, or equivalently if $1 < \rho \leq 2$, then we have $d=2$ and $Q(n,2) = -7n+117 \leq 0$. The same reasoning leads to $Q(n,d) \leq 0$ for $2 < \rho \leq 3$ and $3 < \rho \leq 4$. We now suppose that $\rho > 4$. Writing $d = \rho + t$, with $t\in [0,1]$, and using the fact that $16\rho^3 = n$, we find
    \begin{align*}
        Q(n,d) \leq -n(d+5)+16d^4 & = - 16\rho^3(\rho+t+5) + 16 \big(\rho^4 + 4t\rho^3 + 6t^2\rho^2 + 4t^3\rho +t^4 \big) \\
        & = 16\rho^3(3t-5) + 16 \big(6t^2\rho^2 + 4t^3\rho +t^4 \big) \leq 16R(\rho),
    \end{align*}
    where $R(x) = -2x^3+6x^2+4x+1$. As the coefficients of $R(x+4)$ are all negative, we have $R(x)\leq 0$ for each $x\geq 4$. In particular, $R(\rho)\leq 0$, and we once again obtain $Q(n,d)\leq 0$.
\end{proof}

\section{Proof of the main theorem}
\label{section: proof of main thm}

In this last section we prove our main Theorem~\ref{Thm : main} in the following stronger form.

\begin{Thm}
    \label{Thm : main avec meilleure cst}
    Let $\ee = 0.3748\cdots$ be the unique (positive) solution of the equation $(1+x)e^x = 2$ and set $a = \big(2\ee(2-e^\ee)/9\big)^{1/3}  =  0.3567 \cdots$. There exists an explicit constant $C>0$ such that, for each $n\geq 1$ and any transcendental real number $\xi\in\bR$, we have
    \begin{align*}
        \homega_n(\xi) \leq 2n- a n^{1/3} + C.
    \end{align*}
\end{Thm}

Since $1/3 < a$, it implies Theorem~\ref{Thm : main}. We first establish a preliminary result which uses the following notation. Let $n,d$ be integers with $2\leq d \leq \sqrt{n/4}$. In particular $d\leq 1+n/2$. We define
\[
    \omega(d,n) := 2n+P(n,d), \quad \textrm{where } P(n,d) = \frac{n(4d^2-d-5)+8d^2-2d-15}{2n-8d^2+2d+15},
\]
as well as
\[
    \tau(d,n) = \frac{(2n-d)(n-2d+3)}{\omega(d,n)\big(\omega(d,n)-n-d+3\big)} \AND \mu(d,n):= \frac{n}{(2n-d)\tau}.
\]
Let $\big(\tau_i(d,n)\big)_{0\leq i\leq n/2}$ be the sequence associated to $\tau=\tau(d,n)\in(0,1)$ as in Definition~\ref{Def: tau_j}.

\begin{Thm}
    \label{Thm: main theorem under condition}
    Let $n,d, j$ be non-negative integers with $2\leq d \leq \sqrt{n/4}$ and $1\leq j\leq n/2$. Suppose that
    \begin{align}
        \label{eq: condition sur j}
        \tau_k(d,n) > \mu(d,n) \quad \textrm{for $k=0,\dots,j$}.
    \end{align}
    Then for any transcendental real number $\xi$ we have
    \begin{align}
        \label{eq proof: resultat main a montrer}
         \homega_n(\xi)\leq 2n- \min\big\{d,d_j\big\}, \quad \textrm{where } d_j= 2j-1-\frac{j-1}{\tau_j(d,n)}.
    \end{align}
\end{Thm}

\begin{proof}
    Fix a transcendental real number $\xi$. If $\homega_n(\xi)\leq 2n-d$, then \eqref{eq proof: resultat main a montrer} holds. We now assume that $\homega_n(\xi) > 2n-d$, and we choose a real number $\homega$ such that
    \[
        \homega_n(\xi) > \homega > 2n-d.
    \]
    Let $(P_i)_{i\geq 0}$ denote a sequence of minimal polynomials associated to $n$ and $\xi$ as in Section~\ref{section: minimal pol}. We denote by $(Q_i)_{i\geq 0}$ the sequence of irreducible factors given by Proposition~\ref{prop: existence Q_i}, and denote by
    \[
        m:=m_n(\xi)
    \]
    the dimension of the spaces $\Vect[\bR]{Q_i,Q_{i+1},\dots}$ for each large enough $i$ (as in Definition~\ref{Def m_n}). Proposition~\ref{Prop: estimation omega_n(xi)} yields $\omega_n(\xi)\leq \omega(d,n)$, and by Proposition~\ref{prop: existence Q_i}~\ref{enum: H(Q_i+1) controle par H(Q_i)}, we get, for each large enough $i$,
    \[
        \normH{Q_{i+1}}^{\tau(d,n)} \leq \normH{Q_i}.
    \]
    For simplicity, we write $\tau=\tau(d,n)$ and $\tau_k=\tau_k(d,n)$ for each $k\in\bN$ with $k\leq n/2$. We claim that $j < m-2$. By contradiction, otherwise we have $m-2\leq j\leq n/2$. Then Corollary~\ref{Cor: tau_m_2 < mu} yields $\tau_{m-2} \leq \mu(d,n)$, which contradicts the hypothesis \eqref{eq: condition sur j}. Hence our claim. Let $i\geq 0$. Set $Q = Q_{\sigma_j(i)}$. If $i$ is large enough, there exists a non-zero $P\in\bZ[X]_{\leq n}$ solution of
    \begin{equation*}
        \normH{P} \leq e^{-n} \normH{Q}=: X \AND |P(\xi)| \leq X^{-\homega}.
    \end{equation*}
    By \eqref{eq: Gelfond's Lemma} the (irreducible) polynomial $Q$ does not divide $P$, they are thus coprime. Lemma~\ref{lem : dim V_n+j(P,Q)} implies that $\dim V_{2n-j}(P,Q)\geq 2n-2j+2$. Choose a linearly independent subset
    \[
        \cU_j:=\{U_1,\dots,U_{2n-2j+2}\} \subset \cB_{2n-j}(P,Q)
    \]
    of cardinality $2n-2j+2$. According to \eqref{eq: V_2n-j(A_j[i]) = tout l'espace}, we have $V_{2n-j}(A_j[i]) = \bR[X]_{\leq 2n-j}$. So there exists
    \[
        \cV_j:=\{V_1,\dots,V_{j-1}\} \subset \cB_{2n-j}(Q_{\sigma_j(i)},\dots, Q_i)
    \]
    such that
    \[
        \Vect[\bR]{\cU_j} \oplus \Vect[\bR]{\cV_j} = \bR[X]_{\leq 2n-j}.
    \]
    Then, identifying $\bR[X]_{\leq 2n-j}$ with $\bR^{2n-j+1}$ via \eqref{eq: identification poly avec points de R^m}, we form the generalized determinant
    \begin{align}
        \label{eq proof: det à considérer}
        1 \leq \big|\det(U_1,\dots,U_{2n-2j+2},V_1,\dots,V_{j-1}) \big|.
    \end{align}
    For $k=1,\dots,2n-2j+2$, we have
    \begin{align*}
        \normH{U_k} \ll \normH{Q} \AND |U_k(\xi)| \ll \normH{Q}^{-\homega}.
    \end{align*}
    On the other hand, for $k=1,\dots,j-1$, we have by Eq. \eqref{eq Prop: minoration de Y_j} from Proposition~\ref{Prop: prop resumee avec tau_j}
    \begin{align*}
        \normH{Q} \ll \normH{V_k} \ll \normH{Q_i} \ll \normH{Q}^{1/\tau_j} \AND |V_k(\xi)| \ll \normH{V_k}^{-\homega} \ll \normH{Q}^{-\homega}.
    \end{align*}
    For $i=2,\dots,2n-j+1$, we add to the first row of the determinant \eqref{eq proof: det à considérer} the $i$-th row multiplied by $\xi^{i-1}$. This first row now becomes
    \[
        \big(U_1(\xi),\dots,U_{2n-2j+2}(\xi),V_1(\xi),\dots,V_{j-1}(\xi)).
    \]
    By the above, the absolute value of each of its elements is $\ll \normH{Q}^{-\homega}$. By expanding the determinant, we obtain
    \begin{align*}
        1 \ll \normH{Q}^{2n-2j+1}\normH{Q_i}^{j-1}\normH{Q}^{-\homega} \ll \normH{Q}^{2n-2j+1+(j-1)/\tau_j-\homega}.
    \end{align*}
    By letting $i$ tend to infinity, we deduce that
    \begin{align*}
        \homega \leq 2n-2j+1+(j-1)\tau_j^{-1} = 2n-d_j.
    \end{align*}
    Since $\homega$ may be chosen arbitrarily close to $\homega_n(\xi)$, we finally get \eqref{eq proof: resultat main a montrer}.
\end{proof}

In view of \eqref{eq proof: resultat main a montrer}, the idea is now to choose $d$ and $j$ so that $d$ is maximal and $d \thickapprox d_j$. The next two results aim at simplifying condition~\eqref{eq: condition sur j} of Theorem~\ref{Thm: main theorem under condition}. The second one also provides a simple lower bound for the exponent $\tau_j$.

\begin{Lem}
    \label{lem: condition thm clef seulement pour k=j}
    Let $n,d, j$ be non-negative integers with $2\leq d \leq \sqrt{n/4}$ and $ 1\leq j\leq n/2$. Suppose that $j$ satisfies
    \begin{align}
    \label{eq lem: condition sur k=j seulement pour tau_j}
        \frac{(2n-d)\tau(d,n)^2}{(n-2j)\tau(d,n)+n-j+1} \leq 1 \AND \tau_j(d,n) \geq 0.
    \end{align}
    Then, the sequence $\big(\tau_k(d,n)\big)_{0\leq k \leq j}$ is (strictly) decreasing. In particular, condition~\eqref{eq: condition sur j} is fulfilled if moreover
    \begin{align*}
        \tau_j(d,n) > \mu(d,n).
    \end{align*}
\end{Lem}

\begin{proof}
    Let $\alpha_1 \leq \cdots \leq \alpha_j$ be as in Definition~\ref{Def: tau_j}. Condition \eqref{eq lem: condition sur k=j seulement pour tau_j} is equivalent to $\alpha_j\leq 1$ and $\tau_j(d,n) \geq 0$. By definition, we have
    \[
        \tau_{k-1}(d,n) = \alpha_k^{-1}\tau_k(d,n) + \frac{2k-1}{2n-d} \qquad \textrm{(for $k=1,\dots,j$).}
    \]
    Since $\alpha_k^{-1}\geq \alpha_j^{-1}\geq 1$, this yields $\tau_{k-1}(d,n) > \tau_k(d,n)$. This proves the first assertion of our lemma. The second one follows easily.
\end{proof}

\begin{Lem}
    \label{lem: minorer tau_j}
    Let $n,d, j$ be non-negative integers with $2\leq d \leq \sqrt{n/4}$ and $ 1\leq j\leq n/2$. Define
    \[
        \alpha = \alpha(d,n) := \frac{(2n-d)\tau(d,n)^2}{(n-2)\tau(d,n)+n},
    \]
    and suppose that
    \begin{align}
        \label{eq lem: condition pour minorer tau_j}
        \alpha^{j} > \frac{j(2j-1)\tau(d,n)}{(n-2)\tau(d,n)+n} = \frac{j(2j-1)\alpha}{(2n-d)\tau(d,n)}.
    \end{align}
    Then, $\alpha\in(0,1)$ and for $k=0,\dots,j$, we have
    \begin{align*}
        \tau_k(d,n) \geq \alpha^j\tau(d,n) - \frac{j(2j-1)\tau(d,n)^2}{(n-2)\tau(d,n)+n} > 0.
    \end{align*}
\end{Lem}

\begin{proof}
    We have $\alpha\in(0,1)$ since $\tau(d,n)< 1$ and $d\geq 2$. For simplicity, we write $\tau = \tau(d,n)$. Let $(\sigma_k)_{k\geq 0}$ be the sequence defined by $\sigma_0= \tau$, and
    \begin{align*}
        \sigma_k = \alpha\bigg(\sigma_{k-1} - \frac{2k-1}{2n-d} \bigg) \qquad \textrm{for } k\geq 1.
    \end{align*}
    Using \eqref{eq lem: condition pour minorer tau_j}, we find
    \begin{align}
        \label{eq proof: minoration sigma_j}
        \frac{\sigma_j}{\alpha^j} = \frac{\sigma_{j-1}}{\alpha^{j-1}} - \frac{2j-1}{(2n-d)\alpha^{j-1}} & = \sigma_0 -\frac{1}{2n-d}\sum_{k=1}^{j} \frac{2k-1}{\alpha^{k-1}} \geq \tau -\frac{j(2j-1)}{(2n-d)\alpha^{j-1}} > 0.
    \end{align}
    In particular $\sigma_j\geq 0$. Since $\sigma_{k-1} \geq \alpha^{-1}\sigma_k$, by induction, we get $\sigma_j < \sigma_{j-1} < \cdots < \sigma_0$. Moreover, $\alpha = \alpha_1 \leq \alpha_k$, for each $k\in\bN$ with $1\leq k \leq n/2$, where $\alpha_k$ is as in Definition~\ref{Def: tau_j}. Another quick induction yields $\sigma_k\leq \tau_k$ for $k=0,\dots,j$. We conclude by combining $\sigma_j\leq \sigma_k\leq \tau_k$ with \eqref{eq proof: minoration sigma_j}.
\end{proof}

\begin{proof}[Proof of Theorem~\ref{Thm : main avec meilleure cst}]
    Define a function $f:[0,\infty)\rightarrow \bR$ by $f(x) = x(2-e^x)$. Let $\ee = 0.3748\cdots$ be the unique solution of the equation $(1+x)e^x = 2$. It is the abscissa of the maximum of $f$. Set
    \begin{align*}
        a = \sqrt[3]{\frac{2\ee(2-e^\ee)}{9}} = 0.3567\cdots
    \end{align*}
    Let $n\geq 1$ and define $d=d(n)$ and $j(n)$ by
    \begin{align*}
        d(n) = \lceil a n^{1/3} \rceil \AND j:=j(n):= \bigg\lceil \frac{2\ee n}{9d^2} \bigg\rceil.
    \end{align*}
    We suppose $n \geq 30$ so that $2 \leq d \leq 1+n/2$ and $1\leq j\leq n/2$. Since $d^4/n^2 \asymp d/n = \GrO(n^{-2/3})$, we find $\omega(d,n) = 2n+2d^2 + \GrO(d)$, and then
    \begin{align*}
        \tau(d,n) = 1-\frac{3d^2}{n} +\GrO\bigg(\frac{1}{n^{2/3}} \bigg) \AND \alpha(d,n) =  1-\frac{9d^2}{2n} +\GrO\bigg(\frac{1}{n^{2/3}} \bigg),
    \end{align*}
    (where $\alpha(d,n)$ is defined in Lemma~\ref{lem: minorer tau_j}). In particular, by choice of $j$, we have
    \begin{align}
    \label{eq proof: minoration alpha^j}
        \alpha(d,n)^j = \exp\big(j\log(\alpha(d,n))\big) = \exp\bigg(-\frac{9jd^2}{2n} +\GrO\bigg(\frac{1}{n^{1/3}} \bigg) \bigg) = e^{-\ee} + \GrO\bigg(\frac{1}{n^{1/3}} \bigg).
    \end{align}
    Since
    \[
        \frac{j(2j-1)\tau(d,n)}{(n-2)\tau(d,n)+n} =  \GrO\bigg(\frac{1}{n^{1/3}} \bigg),
    \]
    there exists $N_1\geq 30$ such that condition \eqref{eq lem: condition pour minorer tau_j} of Lemma~\ref{lem: minorer tau_j} is fulfilled for each $n\geq N_1$. Thus, for $k=0,\dots,j$, we have
    \begin{align*}
        \tau_k(d,n) \geq \alpha(d,n)^j\tau(d,n) - \frac{j(2j-1)\tau(d,n)^2}{2n-2} = e^{-\ee} + \GrO\bigg(\frac{1}{n^{1/3}} \bigg).
    \end{align*}
    In particular $d_j = 2j-1-(j-1)/\tau_j(d,n)$ satisfies
    \begin{align*}
        d_j \geq j\big(2- e^{\ee}\big) + \GrO(1) & = \frac{2\ee(2- e^{\ee}) n}{9d^2}  + \GrO(1) = \frac{a^3 n}{d^2} + \GrO(1) = d + \GrO(1).
    \end{align*}
    On the other hand, we have
    \begin{align*}
        \mu(d,n) = \frac{n}{(2n-d)\tau} = \frac{1}{2}+\GrO\bigg(\frac{1}{n^{1/3}} \bigg).
    \end{align*}
    Since $e^{-\ee} > 1/2$, by \eqref{eq proof: minoration alpha^j} there exists $N_2\geq N_1$ such that condition \eqref{eq: condition sur j} of Theorem~\ref{Thm: main theorem under condition} is fulfilled for each $n\geq N_2$. We conclude that for any $n\geq N_2$ and any transcendental real number $\xi$, we have
    \begin{align*}
        \homega_n(\xi) \leq 2n-\min\{d,d_i\} = 2n - d + \GrO(1).
    \end{align*}
\end{proof}

\appendix

\section{Appendix: Twisted heights}
\label{subsection: twisted heights}

The purpose of this appendix  is to give another interpretation of the quantity $\HHstar(V)$ defined in Section~\ref{def petitesse HHstar(V)}. Our first approach was actually to work with the heights $\HH_T$ defined below. We are thankful to Damien Roy for pointed out the link with Hodge's duality.

\medskip

Fix $A\in \GL(\bR^{m+1})$ and let $V$ be a $k$--dimensional subspace of $\bR^{m+1}$ defined over $\bQ$. Its (twisted) height $\HH_A(V)$ is defined as
the covolume of the lattice $A(V\cap\bZ^{m+1})$ inside the subspace $AV$ (with the convention that $\HH_A(V) = 1$ if $V=\{0\}$). Explicitly, we have
\begin{equation}
    \label{eq: def hauteur tordue}
    \HH_A(V) := \norm{A\bx_1\wedge \cdots \wedge A\bx_k },
\end{equation}
where $(\bx_1,\dots,\bx_k)$ is any $\bZ$--basis of the lattice $V\cap \bZ^{m+1}$. Then Schmidt's inequality generalizes as follows
\begin{equation}
    \label{eq : inegalite de Schmidt}
    \HH_A(U+V) \HH_A(U\cap V) \leq \HH_A(U)  \HH_A(V)
\end{equation}
for any subspaces $U,V$ of $\bR^{m+1}$ defined over $\bQ$. The proof is the same as for rational subspaces (see \cite[26, Chapter I, Lemma 8A]{Schmidt} and \cite[§5]{marnat2018optimal}). Similarly to Marnat and Moshchevitin \cite[§5]{marnat2018optimal}, we consider twisted heights of the following form. Let $T>1$ be a parameter. We define the matrix $A_{m,T}\in\GL(\bR^{m+1})$ as
\[
    A_{m,T} = \left(\begin{array}{cccc}
                    T^{m}       &    0    &  \dots &  0  \\
                       0        & T^{-1}  &        &     \\
                       \vdots   &         & \ddots &     \\
                        0       &         &        & T^{-1}
              \end{array}\right)
              \left(\begin{array}{ccccc}
                    1          &  \xi      &  \cdots   &          & \xi^{m}   \\
                    0          &   1       &   0       &  \cdots  &   0       \\
                    \vdots     &           &  \ddots   &          &   \vdots  \\
                               &           &           &          &           \\
                    0          &  \cdots   &           &  0       &    1
              \end{array}\right),
\]
so that for each polynomial $P= a_0+\cdots + a_{m}X^{m} \in \bZ[X]_{\leq m}$ (identified to a point of $\bR^{m+1}$ via \eqref{eq: identification poly avec points de R^m}), we have
\begin{equation}
    \label{eq: expression de Ax pour la hauteur T}
    A_{m,T}\left(\begin{array}{c}
        a_0 \\
        \vdots \\
        a_{m}
    \end{array} \right) =
    \left(\begin{array}{c}
        T^{m}P(\xi) \\
        T^{-1}a_1 \\
        \vdots \\
        T^{-1} a_m
    \end{array} \right).
\end{equation}
We denote by $\HH_{m,T}$ (or simply $\HH_T$ if there is no ambiguity about the integer $m$) the twisted height $\HH_{A}$ associated to the matrix $A=A_{m,T}$. Note that
\[
    \HH_{T}\big(\bR[X]_{\leq m}\big) = \HH_{T}(\bR^{m+1}) = \det(A) = 1.
\]

\begin{Def}
    Let $V$ be a subspace of $\bR[X]_{\leq m}$ defined over $\bQ$. We set
    \begin{align*}
        \HHH(V) = \lim_{T\rightarrow+\infty} T^{-\codim(V)}\HH_{m,T}(V),
    \end{align*}
    where $\codim(V) = m+1-\dim(V)$ denotes the codimension of the space $V$ inside $\bR[X]_{\leq m}$. In particular, $\HHH(\bR[X]_{\leq n}) = 1$, and for any primitive polynomial $P\in \bZ[X]_{\leq m}$, we have
    \[
        \HHH\big(\Vect[\bR]{P}\big) = |P(\xi)| = \HHstar\big(\Vect[\bR]{P}\big).
    \]
\end{Def}

Our goal is now to prove that for any non-zero subspace $V\subset\bR[X]_{\leq m} \simeq \bR^{m+1}$ defined over $\bQ$, we have
\[
    \HHH(V) \asymp \HHstar(V),
\]
where $\HHstar$ is as in Definition~\ref{Def: def HHstar} (and the implicit constant depends on $m$ and $\xi$ only). First, note that since $\dim(U+V) + \dim(U\cap V) = \dim U + \dim V$ for any subspaces $U,V$ of $\bR[X]_{\leq m}$, we deduce from \eqref{eq : inegalite de Schmidt} (with $A=A_{m,T}$) the following version of Schmidt's inequality, which is the analog of Proposition~\ref{Prop: Schmidt inegalite pour HHstar}
\begin{equation}
    \label{eq : inegalite de Schmidt new}
    \HHH(U+V) \HHH(U\cap V) \leq \HHH(U)  \HHH(V),
\end{equation}
valid for any $U,V$ of $\bR[X]_{\leq m}$ defined over $\bQ$.

\begin{Prop}
    \label{Prop: annexe, formule HHH(V)}
    Let $V$ be a $k$--dimensional subspace of $\bR^{m+1}$ defined over $\bQ$, with $1 \leq k\leq m+1$, and set $\Xi_m=(1,\xi,\dots,\xi^m)$. We have
    \begin{equation}
        \label{Prop annexe: equivalence estimate avec hauteur tordue}
        \HHstar(V) \ll \HHH(V) \leq \HHstar(V),
    \end{equation}
    where the implicit constant depends on $\xi$ and $m$ only. Moreover, for any $\bZ$--basis $(\bx_1,\dots,\bx_k)$ of $V\cap\bZ^{m+1}$, we have
    \begin{align}
        \label{eq prop: formule explicite HHH(V)}
        \HHH(V) = \norm{\sum_{i=1}^{k} (-1)^{k-i}\psc{\Xi_m}{\bx_i}\,\bx_1^+\wedge \cdots \wedge\widehat{\bx_i^+}\wedge \cdots \wedge\bx_k^+},
    \end{align}
    where $\bx_i^+\in\bZ^{m}$ denotes the point $\bx_i$ deprived of its first coordinate.
\end{Prop}

Before to prove this result, we introduce some notation that we will need in the proof. Given two positive integers $p$ and $q$, we define $\cI(p,q)$ as the set of $p$--tuples $(i_1,\dots,i_p)$ of integers with $1\leq i_1 < \cdots < i_p \leq q$. Let $\be = (\be_1,\dots,\be_q)$ be the canonical basis of $\bR^q$. For any $I\in\cI(p,q)$ as above, set $\be_{I} = \be_{i_1}\wedge \cdots \wedge \be_{i_p} \in \bigwedge^p\bR^q$. For any $\bX\in\bigwedge^p\bR^q$, we call \textsl{$I$--coordinate of $\bX$} its $\be_I$--coordinate in the basis $(\be_J)_{J\in\cI(p,q)}$. For any $\bx_1,\dots,\bx_p\in\bR^q$, we denote by $M(\bx_1,\dots,\bx_p)$ the $q\times p$ matrix whose columns are $\bx_1,\dots, \bx_p$ written in the basis $\be$, and by $\cD_{I}(\bx_1,\dots,\bx_p)$ the minor formed by the rows of $M(\bx_1,\dots,\bx_p)$ of index $i$ in $I$. Then, writing $\bX = \bx_1\wedge \cdots \wedge\bx_k$, we have the classical formula
\begin{align}
    \label{eq: formule norm multi-vecteur}
    \bX = \sum_{I\in\cI(p,q)}\cD_{I}(\bx_1,\dots,\bx_p) \be_{I} \AND \norm{\bX}^2 = \sum_{I\in\cI(p,q)}\cD_{I}(\bx_1,\dots,\bx_p)^2.
\end{align}
Therefore, for each $I\in\cI(p,q)$, the $I$--coordinate of $\bX$ is $\cD_{I}(\bx_1,\dots,\bx_p)$.

\begin{proof}[Proof of Proposition~\ref{Prop: annexe, formule HHH(V)}]
    Fix $T\geq 1$ and for $i=1,\dots,k$ set
    \begin{align*}
        \bZZ &= \sum_{i=1}^kp_i\bx_1\wedge \cdots \wedge\widehat{\bx_i}\wedge \cdots \wedge\bx_k, \quad  \textrm{where } p_i  = (-1)^{i+1}\psc{\Xi_m}{\bx_i},\\
        \bY & = \lim_{T\rightarrow+\infty}  T^{-m+k-1}\by_1(T)\wedge\cdots\wedge \by_k(T), \quad \textrm{where } \by_i = \by_i(T) = A_{m,T}(\bx_i) \in \bR^{m+1}.
    \end{align*}
    By \eqref{eq: HHstar(V) formule explicite} we have
    \begin{align*}
        \HHstar(V) = \norm{\bZZ} \AND \HHH(V) = \norm{\bY}.
    \end{align*}
    We prove the following properties. For $i=1,\dots,k$ we set $\bz_i = (\psc{\Xi_m}{\bx_i}, \bx_i)\in \bR^{m+2}$.
    \begin{enumerate}[label=(\roman*)]
      \item \label{enum preuve annexe: item 1} For each $J=(1,j_2,\dots,j_k)\in\cI(k,m+2)$, the $J$--coordinate of $\bz_1\wedge\cdots \wedge\bz_k$ is equal to the $K$--coordinate of $\bZZ$, where $K=(j_2-1,\dots,j_k-1)\in\cI(k-1,m+1)$.
    \end{enumerate}
    Fix $I = (i_1,\dots,i_k)\in\cI(k,m+1)$.
    \begin{enumerate}[label=(\roman*)]
        \setcounter{enumi}{1}
      \item \label{enum preuve annexe: item 2} If $i_1 \geq 2$, then the $I$--coordinate of $\bY$ is equal to $0$.
      \item \label{enum preuve annexe: item 3} If $i_1=1$, then the $I$--coordinate of $\bY$ is equal to the $J$--coordinate of $\bz_1\wedge\cdots\wedge\bz_k$, where $J = (1,i_2+1,\dots,i_k+1)$. It is also equal to the $K$--coordinate of $\bZZ$, where $K=(i_2,\dots,i_k)$.
    \end{enumerate}
    To prove the first assertion, it suffices to expand the determinant $\cD_{J}(\bz_1,\dots,\bz_k)$ along its first row. Let $I = (i_1,\dots,i_k)\in\cI(k,m+1)$. Suppose first that $i_1\neq 1$. Then, by Hadamard's inequality, the $I$--coordinate of $\by_1(T)\wedge\cdots \wedge\by_k(T)$ satisfies
    \[
        |D_{I}(\by_1,\dots,\by_k)| \ll \prod_{j=1}^{k} T^{-1}\normH{\bx_j} = \GrO(T^{-k}),
    \]
    and we deduce that the $I$--coordinate of $\bY$ is equal to $0$, which proves assertion~\ref{enum preuve annexe: item 2}. Suppose now that $i_1=1$ and set $J = (1,i_2+1,\dots,i_k+1)$. Then
    \begin{align*}
        D_{I}(\by_1,\dots,\by_k)  = T^{m+1-k}\cD_J(\bz_1,\dots,\bz_k),
    \end{align*}
    hence the first part of~\ref{enum preuve annexe: item 3}. The second part is obtained by combining the above with assertion~\ref{enum preuve annexe: item 1}.

    \medskip

    We deduce from the last two assertions that all the non-zero coordinates of $\bY$ are coordinates of $\bZZ$, thus $\norm{\bY} \leq \norm{\bZZ}$, which proves the second inequality in \eqref{Prop annexe: equivalence estimate avec hauteur tordue}. For the first estimate, we need to estimate the $K$--coordinates of $\bZZ$ with $K\in\cI(k-1,m+1)$ of the form $(1,i_2,\dots,i_{k-1})$. According to assertion~\ref{enum preuve annexe: item 1}, they are exactly the determinants $\cD_J(\bz_1,\dots,\bz_k)$ with $J = (1,2,j_3,\dots,j_k)$ in $\cI(k,m+2)$.

    \medskip

    Fix a $J\in \cI(k,m+2)$ as above. The second row of the matrix $M(\bz_1,\dots,\bz_k)$ is a linear combination of the remaining rows (with coefficients in absolute value between $1$ and $|\xi|^m$). We deduce that $\cD_J(\bz_1,\dots,\bz_k)$ can be written as a linear combination of $\cD_{J'}(\bz_1,\dots,\bz_k)$, where $J'$ belong to the subset of $\cI(k,m+2)$ consisting in the $k$--tuples whose second element is $\geq 3$. By assertion~\ref{enum preuve annexe: item 3}, they are all coordinates of $\bY$, hence $|\cD_J(\bz_1,\dots,\bz_k)|\ll \norm{\bY}$. We conclude that $\norm{\bZZ}\ll \norm{\bY}$.

    \medskip

    Finally, fix $(i_2,\dots,i_k)\in\cI(k-1,m)$ and set $K = (i_2+1,\dots,i_k+1)$. By definition of $\bZZ$, the $K$--coordinate of $\bZZ$ is equal to
    \begin{align*}
        \sum_{i=1}^{k}p_i \cD_I(\bx_1,\cdots,\widehat{\bx_i}, \cdots,\bx_k) = \sum_{i=1}^{k}p_i \cD_J(\bx_1^+,\cdots,\widehat{\bx_i^+}, \cdots,\bx_k^+).
    \end{align*}
    By assertion~\ref{enum preuve annexe: item 3}, this is also the $(1,K)$--coordinate of $\bY$. So, the set of non-zero coordinates of $\bY$ is exactly equal to the set of non-zero coordinates of the point
    \[
        \sum_{i=1}^{k}p_i\,\bx_1^+\wedge \cdots \wedge\widehat{\bx_i^+}\wedge \cdots \wedge\bx_k^+.
    \]
    Eq. \eqref{eq prop: formule explicite HHH(V)} follows from the second identity of \eqref{eq: formule norm multi-vecteur}.
\end{proof}

\Ack

\bibliographystyle{abbrv}


\footnotesize {

}

\Addresses

\end{document}